\documentclass{amsart}
\usepackage{amsmath}
\usepackage{amssymb}
\usepackage{tikz-qtree}
\usepackage[margin=1.2in]{geometry}

\newtheorem{thm}{Theorem}[section]
\newtheorem{lemma}[thm]{Lemma}
\newtheorem{cor}[thm]{Corollary}
\newtheorem{quest}[thm]{Question}

\theoremstyle{definition}
\newtheorem{definition}[thm]{Definition}

\newcommand{\re}{r.e.\ }
\newcommand{\ie}{i.e.\ }
\newcommand{\eg}{e.g.\ }

\numberwithin{equation}{section}

\begin{document}
\title{Non-Splittings of Speedable Sets}

\author{Ellen S. Chih}
\address{Department of Mathematics, University of California, Berkeley, 737 Evans Hall \#3840, Berkeley, CA 94720-3840 USA}
\email{echih@math.berkeley.edu}

\thanks{The author would like to thank Leo Harrington and Theodore Slaman for their helpful comments, insightful discussions, corrections, numerous helpful advice and for carefully reading through a draft of this paper and for their suggestions on presentation and improving clarity. The author would also like to thank Rod Downey for bringing the problem to attention and an anonymous referee for useful suggestions on an earlier version of the paper. The author was partially supported by the National Science Foundation under grant number DMS-1301659.}
\begin{abstract}
We construct a speedable set that cannot be split into speedable sets. This solves a question of B\"{a}uerle and Remmel.
\end{abstract}

\maketitle

\section{Introduction}
According to Blum and Marques \cite{MR0332455}, ``[a]n important goal of complexity theory $\ldots$ is to characterize those partial recursive functions and recursively enumerable sets having some given complexity properties, and to do so in terms which do not involve the notion of complexity." Blum \cite{Blum1967} opened fruitful avenues in this direction when he constructed a $\{0,1\}$-valued total recursive function with arbitrarily large speed-up.

Blum and Marques \cite{MR0332455} expanded the notion of speedability to recursively enumerable (r.e.) sets. An \re set $A$ is \emph{speedable} if for every recursive function, there exists a program enumerating membership in $A$ faster, by the desired recursive factor, on infinitely many integers. Thus, an \re set is nonspeedable if there is an almost everywhere (a.e.) fastest program for it, modulo a recursive factor. Subsequently, Soare \cite{soare77} gave an ``information theoretic" characterization of speedable sets in terms of a well-studied class of \re sets. He proved that a set $A$ is speedable if and only if $A$ is not \emph{semi-low}, namely,
\[\{e:\overline{A} \cap W_e \neq \emptyset \} \nleq_T \emptyset^\prime \]
where $W_e$ denotes the $e^{th}$ \re set.

A \emph{splitting} of an \re set $A$ is a pair of disjoint \re sets $X, Y$ whose union is $A$. Given a property $P$ of a set $A$ (\eg being high, nonrecursive or speedable), it is a natural question to ask whether we can split the set $A$ into a pair $X$, $Y$ such that both $X$ and $Y$ have property $P$. If the property is finitely based or can be made finitely based, it seems that we are able to split the set in a way that still preserves the property. For example, the Friedberg Splitting Theorem \cite{MR0109125} asserts that every nonrecursive set can be split into two nonrecursive sets. Another example is the Sacks Splitting Theorem \cite{MR0146078}, which implies that every nonrecursive set can be split into two nonrecursive sets that are also Turing incomparable. If the property is not finitely based (\eg being high or speedable), there can be obstructions. For example, Lachlan's Nonsplitting Theorem \cite{MR0409150} shows that the Sacks Splitting and Density Theorems cannot be combined.

While working with degrees of bases of \re vector spaces in 1992, B\"{a}uerle and Remmel \cite{MR1274285} raised several questions on the splittings of \re sets. In particular, they asked whether every high \re set can be split into two high \re sets and whether every speedable set can be split into two speedable sets. The first question (on high sets) was negatively answered by Downey and Shore \cite{DS1998}. The second question is the main topic of this paper.

\begin{quest} \cite{MR1274285}
Can every speedable set be split into two speedable sets?
\end{quest}

The answer was thought to be positive for some time. In 1993, Downey, Jockusch, Lerman and Stob \cite{DS1993} proved that every hyper-hyper-simple set can be split into speedable sets, contributing evidence for a positive answer. In 1999, the question was thought to be resolved when Jahn \cite{Jahn1999} published a proof that every speedable set could be split into two speedable sets. His paper was cited \cite{O2004} as a positive case for other splittings of sets with related complexity properties and cited again \cite{CG2003} for introducing various complexity properties and splittings. However, Downey [private communication] pointed out that the proof in \cite{Jahn1999} is incorrect.

The main goal of this paper is to construct a speedable set that cannot be split into speedable sets (\ie to negatively answer Question 1.1). The proof is by a tree construction but there are infinite positive requirements to the left of the true path. In most tree constructions, nodes to the left of the true path are guessing a $\Pi_2$ outcome that is seen to be false so the action to the left settles down and becomes finite. However, in our construction, ``settling down" means an infinite positive $\Pi_1$ action, namely, all $x$ (of some particular type) get put into $B$ quickly.

The first section of the paper sets notation and introduces the main theorem of the paper. The second section examines the case where we are dealing with only one split and the third examines the case where we are dealing with two splits. The proof of the main theorem is given in the fourth section. In the final section, we examine various ways the main theorem could be improved.

\section{Notation and splittings}
Our notation follows \cite{soare}. We use the following conventions.

Unless specified otherwise, all sets are assumed to be recursively enumerable (r.e.). Fix the universal standard enumeration of \re sets as defined in \cite{soare}. Let $W_i$ denote the $i^{th}$ \re set in this enumeration. Let $W_{i,s}$ denote the subset of $W_i$ as enumerated at the end of stage $s$. Let $\Phi_i(x)$ denote the stage $s$ when $x$ enters $W_i$ (\ie the least $s$ such that $x \in W_{i,s}$). We sometimes refer to $\Phi_i$ as $\Phi_V$ where $V = W_i$. For $\alpha, \beta$ that are nodes on our priority tree $T = \Lambda^{<\omega}$, we say that $\alpha$ is to the left of $\beta$, written $\alpha <_L \beta$ if there is some $\gamma \in T$ and $a, b \in \Lambda$ such that $a < b$ and $\gamma^\smallfrown \langle a \rangle  \subseteq \alpha$ and $\gamma^\smallfrown \langle b \rangle \subseteq \beta$ (where $<$ is the ordering on $\Lambda$ and $^\smallfrown$ denotes concatenation). By the true path, we mean the leftmost path travelled through infinitely often in the construction. By ``true" outcome $o$ of $\alpha$, we mean that $o$ is the outcome of $\alpha$ on the true path.

\begin{definition} Let $A$ be an \re set. $A$ is \emph{nonspeedable} if and only if there exists some $i$ such that $W_i = A$ and a recursive function $h$ such that for all $j$,
\[W_j  \subseteq^\ast A \Rightarrow (a.e. x)[x \in A \Rightarrow \Phi_i(x) \leq h(x,\Phi_j(x))] \]
where $\subseteq^\ast$ is subset mod finite and $(a.e. x)$ is all $x$ mod finite.

A set $A$ is \emph{speedable} if and only if it is not nonspeedable.
\end{definition}

By \cite{soare77}, instead of just one enumeration being ``optimal", being nonspeedable also implies that every enumeration is optimal. We use this equivalent condition interchangeably:
\begin{definition} $A$ is \emph{nonspeedable} if and only if for every $i$ such that $W_i = A$ there exists a recursive function $h$ such that for all $j$,
\[ W_j \subseteq^\ast A \Rightarrow (a.e. x)[x \in A \Rightarrow \Phi_i(x) \leq h(x,\Phi_j(x))] \tag{$\ast$} \]
\end{definition}

A \emph{splitting} of an \re set $B$ is a pair of disjoint \re sets $X, Y$ whose union is $B$.

\begin{thm} There is a speedable set $B$ such that if $X$ and $Y$ form a split of $B$, at least one of $X$ or $Y$ is nonspeedable.
\end{thm}

By \cite{soare77}, being nonspeedable is equivalent to being semilow and the following corollary follows immediately:

\begin{cor}
There is a non-semilow set $B$ such that if $X$ and $Y$ form a split of $B$, at least one of $X$ or $Y$ is semilow.
\end{cor}

\medskip
To construct $B = W_i$ to be speedable, we diagonalize against all recursive functions $h$ that could witness ($\ast$); \ie for each recursive $h$, we build a \re $W_j \subseteq B$ that witnesses the failure of ($\ast$) for $h$.

To ensure that $B$ cannot be split into speedable sets, if $X$ and $Y$ form a split of $B$, we first try to ensure that $X$ is nonspeedable. This attempt may interfere with our making $B$ speedable, in which case we see that we have the means to ensure that $Y$ is recursive.

\section{One split}
We first deal with the case where we are given a split of $B=W_e$ recursively in $e$, \ie $W_i, W_j$ are splits of $B$ given by a recursive function $f$ such that $f(e) = (i,j)$.

\begin{lemma} For every recursive function $f$, there is a speedable set $B$ such that if $X, Y$ is the split of $B$ given by $f$, at least one of $X$ or $Y$ is nonspeedable.
\end{lemma}

\begin{proof} We recursively enumerate $B$ and ensure that $B$ is speedable.

The proof is a tree construction in the sense of \cite{soare} (Chapter 14). Nodes work on strategies. For a node $\alpha$ working on a strategy, $\alpha$ builds $B_\alpha$ which is intended to be an enumeration of $B$ with a different timescale. $B_\emptyset$ is $B$. $\alpha$ also builds $P_\alpha$, a recursive set which will be used as a pool of numbers. We refer to $P_\alpha$ as the pool lower priority nodes use. $\alpha$ uses witnesses from the pool given by the previous node. $P_\emptyset$ is $\omega$. A member in a pool can become ``used" but does not necessarily have to go into $B$. One way a number $x$ can become used is if $x$ enters $B_\alpha$. If an number in $\alpha$'s pool is ``unused", $\alpha$ can keep it out of $B$.

Modulo finite injury, a node can recursively tell which pool it is using. A node can also recursively tell which elements in a pool it can keep out.

By the fixed point theorem, we have some index $k$ such that $B = W_k$. We technically want to show that $W_k$ is speedable.

\subsection*{Requirements}
To build a speedable set, we build an \re set $B$ satisfying the following requirements, one for each partial recursive function $h$:
\[Q^\ast_h : (\exists M \subseteq^* B)(\exists^\infty x)(\Phi_B(x) > h(x,\Phi_M(x)),\]

\noindent where $\Phi_B(x)$ is according to when we put $x$ into $B$.

As there is at most a recursive difference between when $x$ enters $B$ and when $x$ enters $W_k$, satisfying $Q_h^\ast$ shows that $W_k$ is speedable.

We do not satisfy $Q_h^\ast$ directly. Instead, we satisfy the following:
\[Q_{h} : (\exists M \subseteq^* B)(\exists^\infty x)(\Phi_B(x) > h(x,\Phi_M(x)) \vee Y \text{ is recursive}.\]

We work on $Q_h$ instead of $Q_h^\ast$ as we are unable to achieve the first disjunct if $X$ is speedable. If $Y$ is recursive, we have another node working on $Q_h$.

We try to build a $M$ satisfying $(\exists^\infty x)(\Phi_B(x) > h(x,\Phi_M(x))$ by putting $x$'s into $M$ and keeping them out of $B$ until stage $h(x,\Phi_M(x)) + 1$. If we succeed for only finitely many $x$'s, we conclude that $Y$ is recursive and thus nonspeedable.

\medskip

To make one of $X$ or $Y$ nonspeedable, we either make $Y$ recursive by some $Q$-strategy or we make $X$ nonspeedable by satisfying the following requirements for the recursive $g$ defined below and all \re sets $V$ (an equivalent to Definition 2.1):
\[S_{V} : (\exists x)(x \in V \wedge x \notin X) \text{ or } (\forall x)[x \in V \Rightarrow \Phi_X(x) \leq g(x, \Phi_V(x))]\]

Define $g(x,s)$ to be the least stage $t$ that $x$ enters $X$ or $Y$ if $x$ is in $B_s$. Otherwise, we let $g(x,s)$ equal to 0. Observe that $g$ is total and recursive.

We also have the following requirements to help the $Q$-strategies (see Lemma 3.8 for where $D_i$'s are required in the verification):
\[D_i: \Delta_i \neq B\]

\subsection*{Remark} It may seem redundant to have the $D$ strategies when every speedable set is not recursive but showing that the set we build, $B$, is speedable requires showing that a particular recursive description does not work. Satisfying all of the $D$ strategies shows that $B$ is not recursive and thus this particular recursive description does not work. It is true that we could incorporate this into the $Q$ strategy but we have chosen to separate the strategies in order to not overly complicate the construction.

\subsubsection*{Outcomes of $Q$-nodes}
A node $\alpha$ working on a $Q$-strategy has outcomes: $\infty$ (for infinitely many), $c$ (for cofinitely many) and $h$. Along the $\infty$ outcome, infinitely many $x$'s in $M_\alpha$ enter $B$ after stage $h(x,\Phi_{M_\alpha}(x)) + 1$. Along the $c$ outcome, cofinitely many $x$'s in $M_\alpha$ enter $B$ before $h(x,\Phi_{M_\alpha}(x)) + 1$ and we see that we have the means to ensure that $Y$ is recursive. Along the $h$ outcome, $h$ is partial.
\subsubsection*{Outcomes of $S$-nodes}
A node $\alpha$ working on a $S$-strategy has outcomes: $k$ (for keep out of $B$) and $s$ (for $(\sharp)$ where $(\sharp)$ will be defined in the description of the strategy). Along the $k$ outcome, either there is some $x$ in both $V$ and $Y$ or $\alpha$ can keep some element in $V$ out of $B$. Along the $s$ outcome, $\alpha$ does not achieve $(\exists x)(x \in V \wedge x \notin X)$ and $\alpha$ tries to achieve $x \in X \Rightarrow \Phi_X(x) \leq g(x, \Phi_V(x))$ for all $x$.
\subsubsection*{Outcomes of $D$-nodes}
A node $\alpha$ working on a $D$-strategy has outcomes: $a$ (for act) and $d$ (for diverge). Along the $a$ outcome, $\Delta_i(x)$ converges and equals 0 and $\alpha$ would like to put $x$ into $B$. Along the $d$ outcome, $\Delta_i(x)$ diverges or $\Delta_i(x)$ converges and does not equal 0 and $\alpha$ tries to keep $x$ out of $B$.
\subsection*{Priority Tree}
\medskip
Fix a recursive ordering of the $Q$- and $S$-requirements.

Let $\Lambda = \{\infty, c, h, k,s,a,d\}$ with ordering $\infty < c < h <k <s <a <d$. The tree is a subset of $\Lambda^{<\omega}$ and is built by recursion as follows:

Assign the highest priority $Q_h$ requirement to the empty node and let $\infty$, $c$ and $h$ be its successors. Assume that we assigned a requirement to $\beta = \alpha \upharpoonright (|\alpha| - 1)$, which we call the $\beta$-requirement. We now assign a requirement to $\alpha$. If $c$ is the last node of $\alpha$, assign $\beta$-requirement to $\alpha$ and let its successors be $\infty$ and $h$.

Suppose $c$ appears in $\alpha$ and is not the last node of $\beta$: If the $\beta$-requirement is a $D$-requirement, assign the highest priority $Q$-requirement that has not been assigned so far and let its successors be $\infty$ and $h$. If the $\beta$-requirement is a $Q$-requirement, assign the highest priority $D$-requirement that has not been assigned so far and let its successors be $a$ and $d$.

Suppose $c$ does not appear in $\alpha$: If the $\beta$-requirement is a $Q$-requirement, assign the highest priority $D$-requirement that has not been assigned so far and let its successors be $a$ and $d$. If the $\beta$-requirement is a $D$-requirement, assign the highest priority $S$-requirement that has not been assigned so far and let its successors be $k$ and $s$. Otherwise, assign the highest priority $Q$-requirement that has not been assigned so far and let its successors be $\infty$, $c$ and $h$.

Observe that we only allow $\infty$ and $h$ to be successors for $Q$-requirement nodes that appear after an instance of $c$. we will prove in the verification (Lemma 3.8) that on the true path, if the $c$ outcome occurs, all $Q_h$-nodes $\alpha$ after it must achieve: $(\exists^\infty x)(\Phi_B(x)> h (x,\Phi_{M_\alpha}(x)) + 1)$ for $M_\alpha$ that $\alpha$ builds if $h$ is not partial.

\subsection*{Dynamics of the construction}
We now describe exactly how the $Q_h$ nodes $\alpha$ and the $S$ nodes achieve their goals. At a typical stage of the construction, we will need to consider the configurations as indicated by Figure 1 below. At $\alpha$, $Q_h$ will attempt to pick a witness, wait for a delay determined by $h$, and put the element into its local version of $B$, giving this to $\beta$ for it to process in a like fashion, assuming that $\beta$ is a $Q_{h^\prime}$ node. (Strictly speaking, this is likely a $\beta^\prime \subset \beta$.) Eventually all such $Q$ nodes process $x$ and it will enter $B$. We refer this as $x$ entering $B$ the \emph{slow} way.

Now, it could be that a $S$-node might care about $x$ and interrupt this procedure. There are two ways this can occur. First some $S$ node $\nu$ of higher priority than $\alpha$ discovers $x \in V_\nu[s]$. Then, $\nu$ can see a global win and will restrain $x$ with priority $\nu$. We will allow this to happen even if $\nu$ is strictly left of the current approximation to the true path. An inductive argument will show that such restrains with a $\nu$-node injuring the action of $\alpha$ with $\nu <_L \alpha$ can happen at most finitely often.

The most important new idea in this construction is that other $S$ nodes can affect the action of $\alpha$. We will allow $\eta$ nodes for $S_{V_i}$ for $\alpha^\smallfrown \langle \infty \rangle \subseteq \eta$ to also pull witnesses from $\alpha$. This will only be allowed if such a node $\eta$ is \emph{active}, meaning that we have actually visited it at some stage $s^\prime <s$. If these $\eta$ nodes prevent $\alpha$ from satisfying the first conjunct of $Q_h$ (i.e.~ cofinitely many $x$'s go into $B$ before stage $h(x,\Phi_{M_\alpha}(x))+1$), we can give a proof that $Y$ is recursive. The construction does the following. If an active $\eta$ sees such an $x$ then $\eta$ \emph{immediately} puts $x$ into $B$ and $\alpha$ makes sure that the next such $x$ must be large. Notice that the hypothesis of the $S$ node being correct implies that $x$ must enter $X$ and not $Y$. Thus, assuming that the $c$ outcome is the true outcome for $\alpha$, no small number can enter $Y$ since we can ensure that pulled numbers that enter $B$ are from a set that is disjoint from $Y$.

The point is the following. The only place (small) numbers that enter $B$ will come from will be from those nodes extending $\alpha^\smallfrown \langle c \rangle$. If $\eta$ above saw these numbers \emph{before} they reached $\alpha$ in their upward climb, then $\eta$ would simply restrain them and win forever. Hence they can only be pulled by $\eta$ after they reach $\alpha$. The relevant $x$ always enters $X$. Then we will reset the sizes of numbers so that we get a recursive description of $Y$.

\begin{figure}
\begin{tikzpicture}[level distance=2cm,sibling distance=1cm,
   edge from parent path={(\tikzparentnode) -- (\tikzchildnode)}]
\Tree [.\node {$\beta$}; [.\node {$\alpha$ ($Q_h$ requirement)};  \edge node[auto=right] {$\infty$};
    [.\node {$S_{V_1}$}; \edge node[auto=right] {$s$};
      [.\node {$S_{V_2}$};]
    ]
    \edge node[auto=left] {$c$};  [.\node {$\Pi_2$ and $\Sigma_2$};] \edge node[auto=left] {$h$};  [.\node {$\gamma$ ($D_i$ requirement)};]    ]]
\end{tikzpicture}
\caption{The dynamics between the $Q$- and $S$-strategies.}
\end{figure}
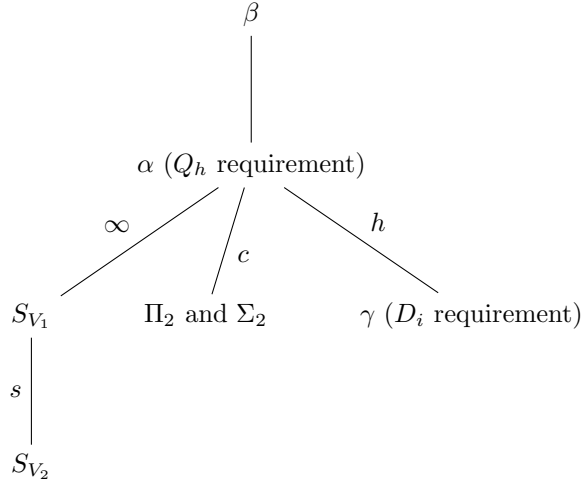

The reader should note that if $\alpha^\smallfrown \langle \infty \rangle$ is \emph{not} on the true path, then we will only visit $\eta$'s extending $\alpha^\smallfrown \langle \infty \rangle$ finitely often so the correct outcome of $\alpha$ is $c$. This outcome has a version of $Q_h$ attached which will succeed in meeting $h$ since we $\emph{know}$ that such $\eta$ cannot pull the element \emph{before} they get to $\alpha$.

It is important to note that there are no $S$ requirements below the $c$ outcome since if $c$ is the true outcome, we have a proof that $Y$ is recursive.
If $\alpha$ is on the true path and $c$ is the outcome of $\alpha$ on the true path, $\Pi_2$ (in Figure 1) reflects the fact that infinitely often, we assign numbers to go into $M_\alpha$ and $\Sigma_2$ (in Figure 1) reflects the fact that only finitely many numbers stay in $M_\alpha$ long enough to see $h$ defined.

As an $S$-strategy can become infinite positive, the $Q$- and $D$-strategies cannot be infinite negative. Furthermore, the $Q$- and $D$-strategies cannot have negative restraint that effect $S$-strategies. Negative restraints that go up and drop down, such as in the minimal pair construction, are also not allowed.

We now turn to the formal details.

\subsection*{Convention} We use the following convention for labeling an element ``used" or ``unused" at $\alpha$ and at stage $s$. If an element becomes used, it is used forever. If an element is picked by a $D$-strategy, it is used. If an element is picked by an $S$-strategy, it is used. If an element goes into $B_\alpha$, it is used. For a $Q$-strategy, a picked element is still ``unused" (see step (1) of the strategy for $Q_h$) but becomes used when put into $M_\alpha$ or into $B_\alpha$ (see step 2(b) of the strategy for $Q_h$).

This convention is needed in the verification (see Lemma 3.4) to prove that for every node $\alpha$, if an element $x$ is unused and in $\alpha$'s pool, $x$ does not go into $B$ while $\alpha$ restrains $x$.

In the following strategies, we work on a node $\alpha$ whose predecessor is $\beta$. $\alpha$ uses the pool given by its predecessor, denoted by $P_\beta$. Recall that the pool of the empty set is $\omega$ and recall that $B_\emptyset = B$.

\subsection*{Strategy for $D_i$} Takes parameters: $(B_\beta, P_\beta)$. Outputs parameters $(B_\alpha, P_\alpha) = (B_\beta, P_\beta \setminus x_\alpha $).

\medskip
Pick an unused witness $x_\alpha \in P_\beta$. Keep $x_\alpha $ out of $B$ and wait for $\Delta_i(x_\alpha )$ to converge to 0. If it does, put $x_\alpha $ into $B_\beta$.

If an element $x$ enters $B_\alpha$, put $x$ into $B_\beta$ unless $D_i$ is assigned to the $h$ outcome, in which case we put $x$ into $B_\eta$ (where $\eta$ is $\beta$'s predecessor). Define $P_\alpha$ to be $P_\beta \setminus x_\alpha$ and $B_\alpha = B_\beta$.

\medskip
Note that $x_\alpha$  may have to enter other sets before entering $B$.

\subsection*{Strategy for $S_{V}$} Takes parameters: $(B_\beta, P_\beta)$. Outputs parameters $(B_\alpha, P_\alpha) = (B_\beta, P_\beta \setminus x_\alpha$) or $= (B_\beta, P_\beta)$.

Below, we go through steps (1) - (4) to find a witness $x_\alpha$ in $V$ that can be permanently be kept out of $X$ either by action of $\alpha$, by action of a higher priority node or by existence of an $x$ already in $V \cap Y$. Step (3)(b) is allows an $S$-node to grab and restrain any element $x$ before $x$ reaches a higher priority node in its upward climb. Step (3)(b) also stops such $x$, if chosen, from continuing in its upward climb.

While $x_\alpha$ has not been defined, go through the following steps in order. Upon defining $x_\alpha$, continue to define $B_\alpha$ and $P_\alpha$ as below.
\begin{enumerate}
\item Ask whether there is some $x$ in $V \cap Y$. If so, let $x_\alpha$ be the least such $x$ and continue to define $P_\alpha$ and $B_\alpha$ as below. Note that $S_V$ is satisfied as $V \nsubseteq X$.
\item Ask whether there is a $\gamma$ working on a $S$-strategy such that $x_\gamma$ is defined and $x_\gamma \in V$. If so, let $x_\alpha = x_\gamma$ and restrain lower priority requirements from putting $x_\alpha$ into $B$. Note that if $\alpha$ permanently restrains $x_\alpha$, $S_V$ is satisfied as $x_\alpha$ witnesses $(\exists x)(x \in V \wedge x \notin B)$, hence $V \nsubseteq X$.
\item Look for a $\gamma$ working on a $D$-strategy with $x_\gamma \in V$ as follows.
\begin{enumerate}
\item Ask whether there is a higher priority $\gamma$ already restraining $x_\gamma \in V$. If so, let $x_\alpha = x_\gamma$. Note that if $\gamma$ permanently restrains $x_\gamma = x_\alpha$, then $S_V$ is satisfied as $x_\alpha$ witnesses $(\exists x)(x \in V \wedge x \notin B)$ hence $V \nsubseteq X$ as in case (2). If $\gamma$ is reset, we reset $x_\alpha$ as well.
\item Ask whether there is some weaker priority $\gamma$  such that $x_\gamma$ is not in $B$ and not in $B_\delta$ for any $\delta$ of higher priority than $\alpha$. If so, let $x_\alpha = x_\gamma$ and restrain lower priority nodes $\eta$ from putting $x_\alpha$ into $B_\xi$ where $\xi$ is $\eta$'s predecessor. Reset the $\gamma$ node (so that if we travel to the $\gamma$ node again, it chooses another witness). Note that if $\alpha$ permanently restrains $x_\alpha$, $S_V$ is satisfied as $x_\alpha$ witnesses $(\exists x)(x \in V \wedge x \notin B)$, hence $V \nsubseteq X$ as in case (2).
\end{enumerate}
\item Ask whether there is an unused $x \in V \cap P_\beta$. If so, let $x_\alpha$ be the least such $x$ and restrain lower priority requirements from putting $x_\alpha$ into $B$. Note that if $\alpha$ permanently restrains $x_\alpha$, $S_V$ is satisfied as $x_\alpha$ witnesses $(\exists x)(x \in V \wedge x \notin B)$, hence $V \nsubseteq X$ as in case (2).
\end{enumerate}

While (1) - (4) do not hold, commit to:
\begin{equation}
\text{When } x \text{ enters } V \text{, put } x \text{ into } B \text{ immediately} \tag{$\sharp$}.
\end{equation}

With $(\sharp)$, if $x$ enters $V$ at stage $s$, it must enter $B$ at stage $s$ so $\Phi_X(x) = g(x,\Phi_V(x))$. Thus, $S_V$ is satisfied (by satisfaction of its second disjunct).

If an element enters $B_\alpha$, put it into $B_\beta$. If $x_\alpha$ is undefined, let $P_\alpha = P_\beta$. Otherwise, let $P_\alpha = P_\beta \setminus x_\alpha$. In either case, let $B_\alpha = B_\beta$.

\subsection*{Strategy for $Q_{h}$} Takes parameters: $(B_\beta, P_\beta)$. Builds parameters: $R_\alpha$, $M_\alpha$, $P_\alpha^\infty$, $B_\alpha^c$. Outputs parameters $(B_\alpha, P_\alpha) = \begin{cases} (B_\beta, P_\alpha^\infty) & \text{ if outcome is } \infty \\
                           (B_\alpha^c, R_\alpha) & \text{ if outcome is } c \\
                           (B_\beta, R_\alpha) & \text{ if outcome is } h \end{cases}$

\medskip
To explain, $R_\alpha$ is a recursive set, $B_\alpha^c$ is an \re set and $P_\alpha^\infty$ is a recursive set.

Recall (see section, Dynamics of the construction) that in the $c$ outcome, we need to ensure that $Y$ is a recursive set. To achieve this, we ensure that $Y$ is contained in a recursive subset of $B$ (namely, $B \setminus R_\alpha$) so $Y$ is the split of a recursive set and thus $Y$ is recursive.

Also recall that in the $c$ outcome, pulled numbers have to enter $X$ and such (small) numbers have to come from nodes below $\alpha^\smallfrown \langle c \rangle$ and reach $\alpha$ in their upward climb. The role $R_\alpha$ plays is to provide a pool of elements for these numbers so pulled numbers only come from $R_\alpha$ (but do not have to be contained in a version of $R_\alpha$ that has been reset). Furthermore, we ensure that $B \cap R_\alpha$ only consists of these pulled numbers, which enter $X$ in the $c$ outcome. Thus in the $c$ outcome, $Y \subseteq B \setminus R_\alpha$ and we will ensure that $B \setminus R_\alpha$ is recursive.

In order to ensure that $B \cap R_\alpha$ only consists of pulled numbers in the $c$ outcome, $R_\alpha$ can be reset as numbers that are not pulled and enter $B$ cannot be in $R_\alpha$. In the $\infty$ outcome, $R_\alpha$ is reset infinitely many times.

The role of $P_\alpha^\infty$ is the following: As the construction tries to ensure that $B \setminus R_\alpha$ is recursive and can potentially reset $R_\alpha$ infinitely often, we need a pool of elements for the $\infty$ outcome, namely, $P_\alpha^\infty$. $P^\infty_\alpha$ will only increase when we see another witness for the $\infty$ outcome so as long as we ensure that $(B \setminus R_\alpha) \cap \overline{P_\alpha^\infty}$ is recursive, $B \setminus R_\alpha$ is recursive in the $c$ outcome as $P_\alpha^\infty$ is finite. In the $\infty$ outcome, $P_\alpha^\infty$ will be recursive and infinite.

\medskip

Now, we explain the intuition behind steps (1) - (4) below: Step (1) builds $R_\alpha$ and step (2) builds $M_\alpha$.

Step (3) looks at the case where $x$ has entered $M_\alpha$ (i.e.~ $x$ has reached $\alpha$ in its upward climb) and no node has pulled $x$. In (3)(a), we put $x$ into $B_\beta$ to let $x$ continue in its upward climb. In 3(b), we add another element to $P_\alpha^\infty$ as we see another instance of $(\exists^\infty x)(\Phi_B(x) > h (x,\Phi_{M_\alpha}(x))+1)$. We also reset $R_\alpha$ as $x$ is not a pulled number but is in $R_\alpha$ and we only want $B \cap R_\alpha$ to contain pulled numbers. We put all elements in $M_\alpha$ into $B$ at this step as these elements are no longer in the $R_\alpha$ after the reset. This is part of our attempt to ensure that $B \setminus R_\alpha$ is recursive.

In step (4), we put all unused elements below $x$ that are in $\overline{R_\alpha} \cap \overline{P_\alpha^\infty}$ into $B$. This step forces later witnesses to be large and in particular, bigger than $x$.

\medskip

Now we formally state steps (1) - (4):

\begin{enumerate}
\item Pick an unused witness $x \in P_\beta$. Put $x$ into $R_\alpha$. $x$ is subject to steps (2) and (3). Note that $x$ does not necessarily have to enter $M_\alpha$.
\item While $x$ is not in $M_\alpha$,
\begin{enumerate}
\item Restrain $x$ from entering $M_\alpha$ until it enters $B_\alpha^c$ (by the action of lower priority nodes).
\item If $x$ goes into $B_\alpha^c$, put $x$ into $M_\alpha$.
\end{enumerate}
\item If $x$ enters $M_\alpha$,
\begin{enumerate}
\item If the current stage is greater than $h(x,\Phi_{M_\alpha}(x))+1$ and $x$ has not entered $B_\beta$, put $x$ into $B_\beta$.
\item If $x$ is successfully is kept out of $B$ until stage $h(x,\Phi_{M_\alpha}(x))+1$, we do the following: Pick an unused witness $y$ in $P_\beta \cap \overline{R_\alpha}$. Put $y$ in $P_\alpha^\infty$ and extend $P_\alpha^\infty$'s recursive description up to $y$, \ie if an arbitrary $z \leq y$ is in $P_\alpha^\infty$ at the current stage, the description says $z$ is in and if $z$ is not in $P_\alpha^\infty$ at the current stage, the description says $z$ is out. Reset $R_\alpha$ and put all elements in $M_\alpha$ into $B$ (since these elements are no longer in our current $R_\alpha$).
\end{enumerate}
\item Put in all unused elements in $\{y: y < x \} \cap \overline{R_\alpha} \cap \overline{P_\alpha^\infty}$ into $B$.
\end{enumerate}

If we have the $c$ outcome, $P_\alpha = R_\alpha$ and $B_\alpha = B_\alpha^c$. If we have the $\infty$ outcome, $P_\alpha = P_\alpha^\infty$ and $B_\alpha = B_\beta$.

Observe that we would like to achieve:

\medskip
($\dagger$) If $x$ enters $M_\alpha$, $x$ is kept out of $B$ until stage $h(x,\Phi_{M_\alpha}(x))+1$.

\medskip
If $(\dagger$) is achieved for infinitely many $x$'s, $Q_h$ is satisfied. Recall we use $\infty$ outcome to denote that infinitely many $x$'s achieve $(\dagger)$. Thus in the $\infty$ outcome, $Q_h$ is satisfied.
\subsection*{Construction}
We call a node working on a $S$-strategy \emph{active} at stage $s$ if the node has acted at some stage $t <s$ and has not been cancelled (or reset).

At stage $s$, define $\delta_s$ (an approximation to the true path) by recursion as follows. Suppose that $\delta_s \upharpoonright e$ has been defined for $e <s$. Let $\alpha$ be the last node of $\delta_s \upharpoonright e$. We now define $\delta_s \upharpoonright (e+1) \supseteq \delta_s \upharpoonright e$.

\begin{enumerate}
\item If $\alpha$ is a $Q$-node, look to see if it is waiting for a number $x$ to go into $B_\beta$ (where $\beta$ is $\alpha$'s predecessor), \ie there is some $x$ in $M_\alpha$ that has not entered $B_\beta$. If so, look to see if an active node $\eta$ to to the left or below $\alpha^\smallfrown \langle \infty \rangle$ would like to put $x$ in (or keep $x$ out if $\eta$ is to the left). Let $\eta$ act. If there are no remaining numbers that $\alpha$ is waiting on, let $\delta_s \upharpoonright (e+1) = (\delta_s \upharpoonright e)^\smallfrown \langle c \rangle$. Otherwise, look to see if there is some $z$ in the remaining numbers such that $h(z,\Phi_{M_\alpha}(z))$ has converged and our current stage is greater than $h(z,\Phi_{M_\alpha}(z))+1$. If so, let $\delta_s \upharpoonright (e+1) = (\delta_s \upharpoonright e)^\smallfrown \langle \infty \rangle$. If not and some $x$ was pulled, let $\delta_s \upharpoonright (e+1) = (\delta_s \upharpoonright e)^\smallfrown \langle c \rangle$. Otherwise, let $\delta_s \upharpoonright (e+1) = (\delta_s \upharpoonright e)^\smallfrown \langle h \rangle$.
\item If $\alpha$ is a $S$-node, look to see if $x_\alpha$ has been defined. If so, let $\delta_s \upharpoonright (e+1) = (\delta_s \upharpoonright e)^\smallfrown \langle k \rangle$. Otherwise, go through steps (1) - (4) in the strategy for $S_V$. If one of (1) - (4) holds and we can define $x_\alpha$, let $\delta_s \upharpoonright (e+1) = (\delta_s \upharpoonright e)^\smallfrown \langle k \rangle$. Otherwise, let $\delta_s \upharpoonright (e+1) = (\delta_s \upharpoonright e)^\smallfrown \langle s \rangle$.
\item If $\alpha$ is a $D$-node, look to see whether $\Delta_i(x_\alpha)$ has converged and equals 0. If so, let $\delta_s \upharpoonright (e+1) = (\delta_s \upharpoonright e)^\smallfrown \langle a \rangle$. Otherwise, let $\delta_s \upharpoonright (e+1) = (\delta_s \upharpoonright e)^\smallfrown \langle d \rangle$.
\end{enumerate}

Reset nodes to the right of $\delta_s$ and let active $S$-nodes to the left of $\delta_s$ or below the $\infty$ outcome for one of the nodes in $\delta_s$ act. At substage $t \leq s$, let $\delta_s \upharpoonright t$ act according to its description in the strategies above, \ie $\delta_s \upharpoonright t$ enumerates numbers in its sets and extends their definitions.

\subsection*{Verification} Recall that the true path denotes the leftmost path traveled through infinitely often by the construction. $\beta$ is $\alpha$'s predecessor unless stated otherwise.

\begin{lemma} For $\alpha$ on the true path working on a $D$-requirement, we only reset $x_\alpha$ at most finitely many times. For $\alpha$ on the true path working on a $Q$-requirement, $M_\alpha \subseteq^\ast B$.
\end{lemma}
\begin{proof}
Let $\alpha$ be a node on the true path. We travel to the left of $\alpha$ at most finitely often in the construction so in the limit, there are only finitely many higher priority active $S$-strategies. Now, consider only the stages after which we do not go to the left of $\alpha$.

Suppose that $\alpha$ is working on a $D$-requirement. At any stage $s$, we only reset $x_\alpha$ due to a higher priority active $S$-strategy. We only reset $x_\alpha$ at most once per every given $S$-strategy due to 3(b) and thus we only reset $x_\alpha$ at most finitely many times.

Suppose that $\alpha$ is working on a $Q$-requirement. We only restrict $x$ in $M_\alpha$ from entering $B$ due to a higher priority active $S$-strategy. We only restrict at most one element per every given $S$-strategy due to 3(b) and thus $M_\alpha \subseteq^\ast B$.
\end{proof}

The next lemma proves that if an $x$ has started on its upward climb, it either reaches $B$ or is reset. The next lemma also shows that if $x$ has started on its upward climb, this action must be initiated by a $D$-node.

\begin{lemma}
If $x_\alpha$ is picked for a node $\alpha$ on the true path and $\alpha$ puts $x_\alpha$ into $B_\beta$, it is either reset or put into $B$. For $\alpha \neq \emptyset$, if $x$ is put in $B_\alpha$ by a lower priority node, $x$ was picked as a witness for a lower priority node $\gamma$ working on a $D$-strategy.
\end{lemma}
\begin{proof}
Suppose that $x_\alpha$ is never reset. Consider only the stages after which we do not go to the left of $\alpha$. By assumption, $x_\alpha$ is put into $B_\beta$. We now prove on the true path that if $x_\alpha$ enters $B_\gamma$ for $\gamma \subseteq \alpha$ then $x_\alpha$ enters $B$ by an $S$-strategy or $x_\alpha$ enters $B_\delta$ where $\delta$ is $\gamma$'s predecessor. Suppose that $x_\alpha$ enters $B_\gamma$ for $\gamma \subset \alpha$. If $x_\alpha$ is grabbed by an active $S$-node to the left, it is put in $B$. If $\gamma$ is working on a $S$-requirement, $\gamma$ either puts $x_\alpha$ into $B$ by $(\sharp)$ or puts $x_\alpha$ into $B_\delta$. If $\gamma$ is working on a $D$-requirement, $\gamma$ puts $x_\alpha$ into $B_\delta$. If $\gamma$ is working on a $Q$-requirement, either $x$ was put into $B_\delta$ as $B_\delta = B_\gamma$ if $\alpha$ extends the $\infty$ or $h$ outcome of $\gamma$ or $\gamma$ puts $x_\alpha$ into $M_\gamma$ as $x_\alpha$ is not reset and thus is not being restrained by another strategy. $\gamma$ then puts $x_\alpha$ into $B_\delta$ at step (3)(a) or $B$ if $x_\alpha \notin R_\alpha$ at step (3)(b) of the $Q$-strategy. The only other possibility is that $x_\alpha$ is grabbed by an active $S$-node to the left and put into $B$. Thus it follows that $x_\alpha$ enters $B$.

By inspection of strategies, if $x$ is put in $B_\alpha$ by a lower priority node, $x$ must come from a lower priority node $\gamma$ working on a $D$-strategy as the $Q$- and $S$-strategies do not pick elements and put them into $B_\eta$ for any $\eta$. The only elements $D$-strategies $\gamma$ put into $B_\eta$ for any $\eta$ are the witnesses $x_\gamma$ they chose.
\end{proof}

\begin{lemma}
For every node $\alpha$, if $x \in P_\beta$ is unused and is not reset as a witness, $x$ does not go into $B$ while $\alpha$ is restraining $x$.
\end{lemma}
\begin{proof}
This follows by inspection of strategies and induction. Let $x$ be an arbitrary unused element in $P_\beta$ and let $s$ be the stage that we choose $x$ to be kept out of $B$. By our convention of unused/used, $x$ is used after stage $s$.

Let $\gamma$ be another node on the tree and let $\delta$ be its predecessor.

If $\gamma$ is working on a $D$-strategy, it only puts $x_\gamma$ or a number put into $B_\gamma$ by its successor into $B_\delta$. $\gamma$ cannot pick $x$ as $x_\gamma$ before stage $s$ as $x$ would then be a used witness at stage $s$. Thus, $x_\gamma \neq x$. $\gamma$ cannot pick $x$ as $x_\gamma$ at some stage $t >s$ as $x_\gamma$ must be unused when chosen. The same reasoning shows that $x$ cannot be put into $B_\gamma$ by its successor as such a number comes from a lower priority $D$-node by the previous lemma.

If $\gamma$ is working on a $S$-strategy, it only puts elements into $B$ due to the $(\sharp)$ requirement. If $x_\alpha \in V$, $\gamma$ would either reset $x_\alpha$ by step (2) or (3) in $S$'s strategy, set $x_\gamma = x_\alpha$ or have already defined $x_\gamma \neq x_\alpha$.

If $\gamma$ is working on a $Q$-strategy, it either puts elements into $B$ at step (4) or puts elements from $M_\gamma$ into $B_\delta$ or $B$ at step (3). At step (4), $\gamma$ only puts unused elements into $B$ to make $B$ recursive on the complement of $R_\gamma$. If an arbitrary $y$ goes into $M_\gamma$, it must enter $B_\gamma$ first. By the previous lemma, $y$ must be a $x_\eta$ for a lower priority $D$-node $\eta$. $\eta$ cannot pick $x$ as $x_\eta$ before stage $s$ as $x$ would be a used witness at stage $s$. Thus $x_\eta \neq x$ by assumption. $\eta$ cannot pick $x$ at some stage $t >s$ as $x_\eta$ must be unused when chosen.
\end{proof}

\begin{lemma} Let $\alpha$ be a node on the true path and let $s$ be the least stage such that the true path does not go to the left of $\alpha$ after stage $s$. $P_\alpha$ is either infinite at stage $s$ or for infinitely many stages $t >s$, there are be unused elements added to $P_\alpha$ at stage $t$.
\end{lemma}
\begin{proof}
By induction and inspection of strategies.
If $\alpha$ is working on a $D$-requirement or $S$-requirement, $P_\alpha = P_\beta \setminus x_\alpha$  or $= P_\beta$.

If $\alpha$ is working on a $Q$-requirement and its outcome is $c$, $P_\alpha = R_\alpha$. $\alpha$ adds elements to $R_\alpha$ infinitely often at step (1) and does not reset $R_\alpha$ after $s$. If the outcome for $\alpha$ is $\infty$, $\alpha$ adds an element to $P_\alpha^\infty$ infinitely often at step (3)(b). Since $P_\alpha = P_\alpha^\infty$ for the $\infty$ outcome, we are done.
\end{proof}

\begin{lemma}Each $D_i$ requirement is satisfied.
\end{lemma}
\begin{proof}
Let $\alpha$ be a node on the true path that is working on the $D_i$ requirement and by Lemma 3.2, let $s$ be the least stage such that $x_\alpha$ is not reset.
If $\Delta_i(x_\alpha)$ does not converge or does not equal 0, the $x_\alpha$ does not enter $B$ by Lemma 3.4 and the requirement is satisfied.

Now assume that $\Delta_i(x_\alpha)$ converges and is equal to 0. In the construction, the $\alpha$-strategy puts $x_\alpha$ into $B_\gamma$ where $\gamma$ is its predecessor. As $x_\alpha$ does not reset after stage $s$, $x_\alpha$ is put into $B$ by Lemma 3.3 and the requirement is satisfied.
\end{proof}

\begin{lemma} Either every $S_V$ requirement is satisfied (and $X$ is nonspeedable) or $Y$ is nonspeedable.
\end{lemma}
\begin{proof}
Suppose that $c$ appears on the true path. Let $\alpha$ be the node on the true path whose outcome is $c$ and let $Q_h$ be the requirement that $\alpha$ is working on. As $\alpha$'s outcome is $c$, cofinitely many $x$'s in $M_\alpha$ enter $B$ before $h(x,\Phi_{M_\alpha}(x))+1$. Thus, higher priority requirements $S_{V_i}$ must grab cofinitely many elements of $M_\alpha$ and put them into $B$ before $h(x,\Phi_{M_\alpha}(x)) + 1$ due to ($\sharp$). In this case, we have the following equation: $B \cap R_\alpha = M_\alpha \cap R_\alpha = (V_0 \cup \cdots \cup V_i) \cap R_\alpha = X \cap R_\alpha$ (modulo finitely many elements). The last equality comes from (1) in the strategy for $S_V$ (\ie we commit to $(\sharp)$ only if there is no $x$ in both $V$ and $Y$). By inspection of the $Q_h$ strategy, we see that $B$ is made recursive outside of $R_{\alpha}$ by step (4). $Y$ is a split of this recursive part of $B$ and thus it is recursive. Therefore $Y$ is nonspeedable.

Suppose that $c$ does not appear on $f$. By construction of the tree, for every $S_V$ requirement, there is a node on the true path working it. By inspection of the strategy for $S_V$, either one of (1) - (4) holds or it commits to $(\sharp)$ for cofinitely many stages in the construction. If it commits to $(\sharp)$, we have that whenever $x$ enters $V$, $x$ enters $B$ immediately. Therefore $g(x,s) = g(x, \Phi_V(x)) = \Phi_X(x)$ by the definition of $g$ and thus $S_V$ is satisfied.

If (1) holds, $S_V$ is satisfied as $V \cap Y \neq \emptyset$ and $X \cap Y = \emptyset$. Thus $V \nsubseteq X$.

If (2) or (3) holds, a number in $V$ is permanently kept out of $B$ by the proof of Lemma 3.5. Thus $V \nsubseteq B$ and so $V \nsubseteq X$ and the requirement is satisfied.

If (4) holds, a number in $V$ is permanently kept out of $B$ by Lemma 3.5. Thus $V \nsubseteq B$ and so $V \nsubseteq X$ and the requirement is satisfied.
\end{proof}

\begin{lemma} Each $Q_h$ requirement is satisfied.
\end{lemma}
\begin{proof}
Let $\alpha$ be the first node on the true path $f$ that is working on $Q_h$.

If $c$ does not appear on $f$, each node $\alpha$ that is working on a $Q_h$ strategy has outcome $\infty$ (or $h$) and thus there are infinitely many $x$'s that are kept out of $B$ until stage $h(x,\Phi_{M_\alpha}(x))+1$ or $h$ is partial. Therefore, each $Q_h$ is satisfied.

Suppose that $c$ appears on $f$. If $\alpha$ appears before the $c$ outcome and its successor is not $c$, we know its outcome is $\infty$ or $h$ and thus its requirement is satisfied. If $\alpha$'s outcome is $c$, the next node $\gamma$ on the true path works on the same requirement $Q_h$. Now suppose $x$ enters $M_\gamma$ at stage $s$. Such an $x$ must exists. If no $x$ ever enters $M_\gamma$ from some point on, $B$ would be recursive contradicting Lemma 3.6. Let $\eta \leq_L \alpha$ be working on a $S$-requirement. $\eta$ cannot lie below the $c$ outcome by construction of the priority tree. If $\eta$ lies to the left or above the $c$ outcome, $\eta$ either restrains $x$ permanently or has already restrained another element out permanently by (3)(b) in the strategy for $S_V$. By the proof of Lemma 3.3, there are only finitely many such active $\eta$'s that restrain numbers from $M_\gamma$ and each $\eta$ only restrains one number. If $x \in M_\gamma$ eventually enters $B$, it enters $B$ at stage $\geq$ the stage it enters $B_\delta$ (where $\delta$ is $\gamma$'s predecessor) by Lemma 3.3. Thus $\Phi_B(x) > h(x,\Phi_{M_\gamma}(x))$ for infinitely many $x$'s. By Lemma 3.2, $M_\gamma \subseteq^\ast B$ so $Q_h$ is satisfied. If $\alpha$ appears after $\gamma$, the same reasoning as the previous case shows that $\Phi_B(x) > h(x,\Phi_{M_\alpha}(x))$ for infinitely many $x$'s and $M_\alpha \subseteq^\ast B$ by Lemma 3.2.
\end{proof}

This ends the verification for the one split case.

In the one split case, we are either successful at making $X$ nonspeedable by satisfying $S$-requirements or successful at concluding that $Y$ is nonspeedable from the failure of an $Q_h$ at satisfying its requirement using $M_\alpha$. When working with more than one split, we would like to conclude that one of $Y_0$, $Y_1$,... is nonspeedable in the latter case. However, there could be $V_i$'s for different pairs of splits (\eg we could have various $V_i$'s for $X_0$ and various $V_j$'s for $X_1$) so the equation in $R_\alpha$ could involve more than one $X$ (\eg $X_0$, $X_1$, $X_2$).

To overcome this difficulty, we switch which side of the split we are making nonspeedable (\eg from $X_i$ to $Y_i$). Let $Z_i$ be the side of the split we are making nonspeedable. In the case that $\alpha$ fails to satisfy its requirement, switching $Z_i$ will give us some progress on higher priority requirements. In the proof of the one split case, having the equation: $B \cap R_\alpha = M_\alpha \cap R_\alpha = (V_0 \cup \cdots \cup V_i) \cap R_\alpha = X \cap R_\alpha$ (modulo finitely many elements) gave us progress on higher priority requirements by allowing us to conclude that $Y$ is nonspeedable. As equations of such form are important for showing that we have made progress on higher priority requirements for the two split and general case, we give the following definition:

\begin{definition} By \emph{equation in $R_\alpha$}, $B = M_\eta = V_{i_0} \cup ... \cup V_{i_k} = Z_{l_0} \cup ... \cup Z_{l_n}$ (where $\alpha$, $\eta$ are arbitrary nodes and $i_0$,...,$i_k$ and $l_0,...,l_n$ are arbitrary numbers), we mean that the following holds: $B \cap R_\alpha = M_\eta \cap R_\alpha = (V_{i_0} \cup ... \cup V_{i_k})  \cap R_\alpha = (Z_{l_0} \cup ... \cup Z_{l_n}) \cap R_\alpha$.
\end{definition}
If the context is clear, we may drop ``in $R_\alpha$.".

\end{proof}

\section{Two splits}
In this section, we work with the case where we are recursively given two splits: $X_0, Y_0$ and $X_1, Y_1$. We refer to $Z_i$ as half of the split we are making nonspeedable.

Recall that in the one split case, we are either successful at making $X$ nonspeedable by satisfying $S$-requirements or successful at concluding that $Y$ is nonspeedable from the failure of an $Q_h$ at satisfying its requirement using $M_\alpha$ and the equation in $R_\alpha$ (see Definition 3.9) that its failure implies. However when working with two splits, there could be $V_i$'s for different pairs of splits (\eg we could have various $V_i$'s for $X_0$ and various $V_j$'s for $X_1$) so the equation in $R_\alpha$ could involve $X_0$ and $X_1$.

To overcome this obstruction, we use the following idea: If $M_\alpha$ fails to satisfy its requirement (and we have an equation in $R_\alpha$ as before \eg $X_0 \cup X_1 = B$), we switch $Z_1$ from $X_1$ to $Y_1$. If another $M_\gamma$ fails to satisfy its requirement where $\gamma \supseteq \alpha$, we have an equation in $R_\gamma$, \eg $X_0 \cup Y_1 = B$. We also switch $Z_0$ from $X_0$ to $Y_0$ and reset $Z_1$ back to $X_1$. From these equations together and since $R_\gamma \subseteq R_\alpha$, we can conclude the following equation in $R_\gamma$: $X_0 = B$. If we continue to see more nodes $\nu \supseteq \gamma$ fail to satisfy its requirement using $M_\nu$, our method of switching the value of $Z_i$ lets us conclude $Y_0 = B$ in some $R_\eta \subseteq R_\gamma$ (for some node $\eta$). This will allow us to prove that only a fixed number of switchings can occur. These switchings will be explained in detail in the proof.

\begin{lemma} For any two recursive functions $f, g$, there is a speedable set $B$ such that if $X_0, Y_0$ is the split of $B$ given by $f$ and $X_1, Y_1$ is a split of $B$ given by $g$ then for each $i$, at least one of $X_i$ or $Y_i$ is nonspeedable.
\end{lemma}

\begin{proof} The proof is by a tree construction like in the one split case.
\subsection*{Requirements}
To make $B$ speedable, we have similar requirements $Q_h$ as in the one split case:
\[Q_{h} : (\exists M \subseteq^* B)(\exists^\infty x)(\Phi_B(x) > h(x,\Phi_M(x)) \vee \text{ switch some } Z_i \]
Switching some $Z_i$ is a form of progress on higher priority strategies just as concluding that $Y$ is nonspeedable was a form of progress for the one split case.

To make $X_i$ or $Y_i$ nonspeedable, we have the following requirements:
\[S_i : (\exists g_i)(\exists x)(x \in V \wedge x \notin Z_i) \text{ or } [x \in Z_i \Rightarrow \Phi_{Z_i}(x) \leq g_i(x,\Phi_V(x))] \]
and its subrequirements, where $g_i$ is as defined below:
\[S_{i, V}: (\exists x)(x \in V \wedge x \notin Z_i) \text{ or } [x \in Z_i \Rightarrow \Phi_{Z_i}(x) \leq g_i(x,\Phi_V(x))] \]
Let $g_i(x,s)$ be the least $t$ that $x$ enters $X_i$ or $Y_i$ if $x$ is in $B_s$. Otherwise, we define $g_i(x,s)$ to be 0. Observe that $g$ is total and recursive. We also require $S_i$ to have a higher priority than $S_j$ iff $i <j$. We occasionally drop the $i$ subscript if the context is clear.

We refer to $S_i$ as the parent requirement of $S_{i,V}$.

Again, we have the $D_i$ requirements as before to help the $Q$-requirements.

\subsection*{Strategies}

The strategies for $Q$-, $S$- and $D$-requirements are the same as before except for the following differences:
\begin{enumerate}
\item Each $\alpha$ node has a current $Z_i^\alpha$ that it is working on. Each $Z_i^\alpha$ is initially defined to be $X_i$ but may switch to $Y_i$ in the course of the construction.
\item The strategy for $S_{i,V}$ is the same as the strategy for $S_V$ in the one split case except that instead of $X$ and $Y$, we have $Z_i^\alpha$ and the other half of the $i^{th}$ pair.
\end{enumerate}

\subsection*{Outcomes}
Each node $\alpha$ working on a $D$-strategy has outcomes: $a$ (for acts) and $d$ (for diverge). On the $a$ outcome, $\Delta_i(x)$ converges and equals to 0 and $\alpha$ would like to put $x$ into $B$. Along the $d$ outcome, either $\Delta_i(x)$ diverges or is not equal to 0.

\medskip
Each node $\alpha$ working on a $Q$-strategy has outcomes: $\infty$, $c$ and $h$. Along the $\infty$ outcome, infinitely many $x$'s in $M_\alpha$ enter $B$ after stage $h(x,\Phi_{M_\alpha}(x))+1$. Along the $c$ outcome, cofinitely many $x$'s in $M_\alpha$ enter $B$ before $h(x,\Phi_{M_\alpha}(x))+1$ and we switch $Z_i$. Along the $h$ outcome, we see that $h$ is partial.

\medskip

Each node $\alpha$ working on a $S_i$-strategy has a blank outcome. In the general case, it is important to keep track of the overall timing of when $x$'s from $V$'s enter $B$ due to actions of $S_{i,V}$ strategies. Here, $S_i$ keeps track of the state of its child nodes.

Each node $\alpha$ working on a $S_{i,V}$-strategy has outcomes: $k$ and $s$. Along the $k$ outcome, either there is some $x$ in both $V$ and $Y$ or $\alpha$ can keep some element in $V$ out of $B$. Along the $s$ outcome, $\alpha$ cannot achieve $(\exists x)(x \in V \wedge x \notin X)$ and we try to achieve $x \in X \Rightarrow \Phi_X(x) \leq g(x, \Phi_V(x))$ for all $x$.

\subsection*{Switchings of $Z_0^\alpha$ and $Z_1^\alpha$}
\begin{figure}
\begin{center}
   \begin{tabular}{| c || c | c | c | c |}
    \hline
    Configuration & Node & Value of $Z_0$ & Value of $Z_1$ & Equation in $R_{node}$ \\ \hline
    0 & before $\beta$ & $X_0$ & $X_1$ & no equation by assumption \\ \hline
    1 & $\beta$ & $X_0$ & $Y_1$ & $X_0 \cup X_1 = B$ \\ \hline
    2 & $\alpha$ & $Y_0$ & $X_1$ & $X_0 \cup Y_1 = B$ \\ \hline
    3 & $\delta$ & $Y_0$ & $Y_1$ & $Y_0 \cup  X_1 = B$ \\ \hline
    4 & $\gamma$ & $Y_0$ & $Y_1$ & $Y_0 \cup Y_1 = B$ \\ \hline
    \end{tabular}
\end{center}
\caption{A table summarizing the possible switching that can occur for the two split case. The ``Value of $Z_0$ (or $Z_1$)" for configuration $n$ refers to the value of $Z_0^\nu$ ($Z_1^\nu$) for nodes $\nu$ between the node for configuration $n$ and the node for configuration $n + 1$. An equation in $R_{node}$ for configuration $n$ refers to $R_\eta$ where $\eta$ is the value of the node for configuration $n$.}
\end{figure}
\begin{figure}
\begin{tikzpicture}[level distance=2cm,
   edge from parent path={(\tikzparentnode) -- (\tikzchildnode)}]
\Tree [.\node {$\beta$ with $c$ outcome}; \edge[dashed] node[auto=right] {Nodes in between obey Config. 1.}; [.\node {$\alpha$ with $c$ outcome}; \edge node[auto=right] {$\infty$};
    [.\node {$S_{V_1}$ with Config. 1}; ]
     \edge node[auto=right] {$c$}; [.\node {$\Pi_2$ and $\Sigma_2$}; \edge[dashed] node[auto=right] {Nodes obey Config. 2}; [.\node {$\delta$ with $c$ outcome}; \edge[dashed] node[auto=right] {Nodes obey Config. 3}; [.\node {$\gamma$}; ]]]\edge node[auto=left] {$h$};  [.\node {node with Config. 1};] ]]]
\end{tikzpicture}
\caption{In this figure, we assume that there is no node with true outcome $c$ above $\gamma$ except for $\beta, \alpha$ and $\delta$. Assuming that this figure signifies a portion of the true path, the dashed edges between nodes represents the true path in between the two nodes.}
\end{figure}
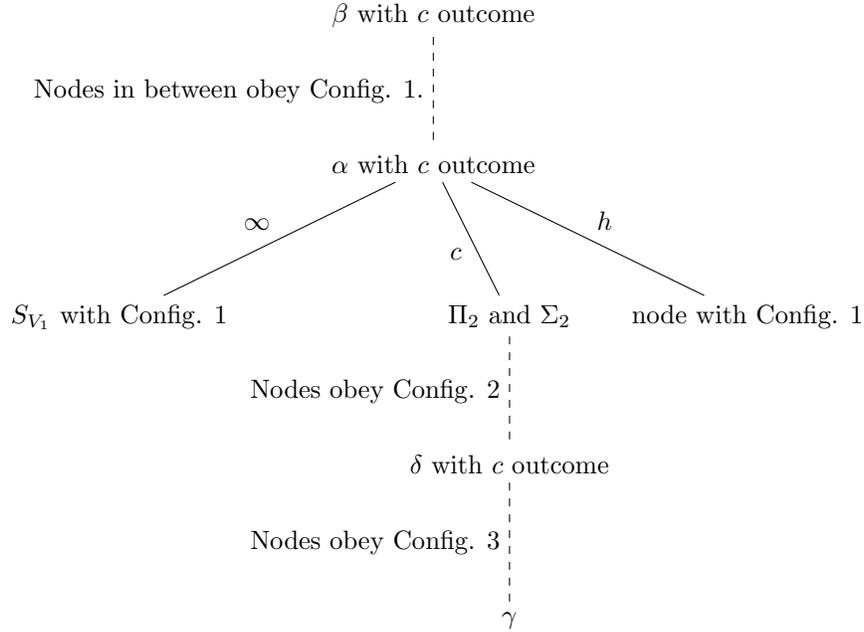
The new ingredient in the two split case that goes beyond the one split case is our way of assigning $Z_i^\alpha$ to each node $\alpha$ so that if we get an equation in $R_\alpha$ from the $c$ outcome, the equation in $R_\alpha$ gives us progress on higher priority strategies. In the one split case, we proved that we could eventually get $Y$ to be recursive. Here, the situation is more complicated and we switch the $Z_0$ and $Z_1$ so that we can obtain a contradiction from getting too many $c$ outcomes for different nodes on the same path.

The main idea is the following: A $c$ outcome for a $Q$-node causes a switch of the highest indexed $Z_i$ that is equal to $X_i$ and resets all higher indexed (for $k > i$) sets $Z_k$ to $X_k$. If such a $Z_i$ does not exist, we do not switch anything. We will prove in the verification that such a $Z_i$ will always exist for $c$ outcomes on the true path (see Lemma 4.3). Switching at some node $\alpha$ effects $Z_i^\gamma$ for $\gamma \supseteq \alpha$ unless a switching occurs at a later node $\nu \supseteq \alpha$.

Now, we assume that $\beta \subseteq \alpha \subseteq \delta \subseteq \gamma$ are on the true path with true outcome $c$ (as in Figure 3 above). We also assume that there is no other $Q_h$ requirement node with true outcome $c$ between them and $\beta$ is the first node on the true path working on a $Q_h$-requirement with true outcome $c$ and so $Z_0^\beta$ and $Z_1^\beta$ are equal to their initial values, $X_0$ and $X_1$ respectively. In other words, we assume that $\beta, \alpha, \delta, \gamma$ are the first, second, third and fourth node (respectively) on the true path with true outcome $c$.  The possibilities are the following (summarized in Figure 2 above):
\begin{enumerate}
\item As $c$ is the true outcome of $\beta$, we have the following equation in $R_\beta$ as in the one split case: $X_0 \cup X_1 = B$. We switch $Z_1$ from $X_1$ to $Y_1$.
\item As  $c$ is the true outcome of $\alpha$, we have the following equation in $R_\alpha$ as in the one split case: $X_0 \cup Y_1 = B$. We switch $Z_0$ from $X_0$ to $Y_0$ and reset $Z_1$, \ie switch $Z_1$ from $Y_1$ back to $X_1$.
\item As  $c$ is the true outcome of $\delta$, we have the following equation in $R_\delta$ as in the one split case: $Y_0 \cup X_1 = B$. We switch $Z_1$ from $X_1$ to $Y_1$.
\item (cannot occur on the true path) As $c$ is the true outcome of $\gamma$, we have the following equation in $R_\gamma$ as in the one split case: $Y_0 \cup Y_1 = B$. We do not switch anything.
\end{enumerate}

Possibility (4) cannot occur due to the following reasoning: Suppose for a contradiction that possibility (4) occurs. As $\beta \subseteq \alpha \subseteq \delta \subseteq \gamma$, we have $R_\beta \supseteq R_\alpha \supseteq R_\delta \supseteq R_\gamma$. Therefore,  we have the following equations in $R_\gamma$: $X_0 \cup X_1 = B$, $X_0 \cup Y_1 = B$, $Y_0 \cup X_1 = B$ and $Y_0 \cup Y_1 = B$. The first two equations implies the following equation in $R_\gamma$: $X_0 = B$. The last two equations implies the following equation in $R_\gamma$: $Y_0 = B$. The satisfaction of some $D_i$ requirement on a node after $\gamma$ puts some element $x$ into $B \cap R_\gamma$ as requirements being satisfied by nodes on the true path after $\gamma^\smallfrown \langle c \rangle$ only take witnesses from $R_\gamma$. However, this would mean that $x$ enters both $X_0$ and $Y_0$ by the two equations in $R_\gamma$: $X_0 = B$ and $Y_0 = B$. This is a contradiction because $X_0$ and $Y_0$ are disjoint as $X_0$ and $Y_0$ form a split of $B$.
\subsection*{Priority Tree}
Fix a recursive ordering of the $D$-, $Q$- and $S$-requirements, respectively, where $S_i$ is of higher priority of than $S_{i,V}$.

Let $\Lambda = \{\infty, c, h,k,s,a,d\}$ with ordering $\infty < c < h<k <s <a <d$. The tree is a subset of $\Lambda^{<\omega}$.

We define by recursion a function $G$ such that $G$ assigns to each $\alpha$ the list of $Z_0^\alpha$ and $Z_1^\alpha$ it is working on in the construction (\eg $G(\alpha) = (X_0, Y_1)$). We also keep track of two lists $L_1(\alpha)$ and $L_2(\alpha)$ for each node. $L_1$ keeps track of the $S$-requirements that have appeared so far (but gets reset at every appearance of a $c$ outcome) and $L_2$ keeps track of $S$-requirements that need to be repeated and we remove a requirement from this list once it has been repeated.

We assign requirements to nodes as well as define $G, L_1, L_2$ by recursion at the same time.
For the empty node, assign the highest priority $Q_h$ requirement to $\alpha$. Let $G(\emptyset) = (X_0, X_1)$ and let $L_1(\emptyset) = L_2(\emptyset) = \emptyset$. Suppose that we have assigned a requirement to $\beta = \alpha \upharpoonright (|\alpha| -1)$, which we will call the $\beta$ requirement and defined $G(\beta), L_1(\beta)$ and $L_2(\beta)$. We now assign a requirement to $\alpha$ and define $G(\alpha), L_1(\alpha)$ and $L_2(\alpha)$.

\subsubsection*{Defining $G(\alpha)$}
Ask whether $\beta$ is a node working on some $Q_h$ with successor $c$. If not, let $G(\alpha)$ be $G(\beta)$. If so, let $k$ be $1$ if $X_1$ appears in $G(\beta)$ and let $k$ be $0$ otherwise. We have the following possibilities:

\begin{enumerate}
\item If $k = 0$ and $Y_0$ appears in $G(\beta)$, define $G(\alpha)$ to be $G(\beta)$.
\item If $k = 0$ and $X_0$ appears in $G(\beta)$, define $G(\alpha)$ to be $(Y_0, X_1)$.
\item If $k = 1$, define $G(\alpha)$ to be $G(\beta)$ with $X_1$ switched to $Y_1$, \ie if $G(\beta) = (Z^\prime, X_1)$ then define $G(\alpha)$ to be $(Z^\prime, Y_1)$.
\end{enumerate}

\subsubsection*{Defining $L_1(\alpha)$ and $L_2(\alpha)$}
Ask whether $\beta$ is a node working on some $Q_h$ with successor $c$. If so, let $L_2(\alpha) = L_1(\beta)$ and set $L_1(\beta) = \emptyset$. Otherwise, ask whether $\beta$ is a node working on some $S$-requirement. If so, let $L_1(\alpha)$ be $L_1(\beta)$ with $\beta$'s requirement affixed at the end. If $L_2(\beta)$ is not empty and this $S$-requirement is in $L_2(\beta)$, let $L_2(\alpha)$ be $L_2(\beta)$ with this $S$-requirement removed. Otherwise, let $L_1(\alpha) = L_1(\beta)$ and let $L_2(\alpha) = L_2(\beta)$.

\subsubsection*{Assigning a requirement to $\alpha$}
If the $\beta$-requirement is an $Q$-requirement with outcome $c$, assign the $\beta$-requirement to $\alpha$ and let its successors be $\infty$, $c$ and $h$.

Otherwise, check if $L_2(\alpha)$ is empty or not. If $L_2(\alpha)$ is not empty, assign the highest priority $S$ requirement in $L_2(\alpha)$ to $\alpha$. If this requirement is a $S_i$-requirement, let its successor be the blank outcome and if this requirement is a $S_V$-requirement, let its successors be $k$ and $s$.

If $L_2(\alpha)$ is empty, we assign a requirement to $\alpha$ based on the type of requirement assigned to $\beta$-requirement: If the $\beta$-requirement is a $Q$-requirement, assign the highest priority $D$-requirement that has not been assigned so far and let its successors be $a$ and $d$. If the $\beta$-requirement is a $D$-requirement, assign the highest priority $S$-requirement that has not been assigned so far and let its successors be $k$ and $s$ if it is a $S_V$ requirement and let its successor be the blank outcome if it is a $S_i$ requirement. Otherwise, assign the highest priority $Q$-requirement that has not been assigned so far and let its successors be $\infty$, $c$ and $h$.

\subsection*{Construction}
We call a node working on a $S$-strategy \emph{active} at stage $s$ if the node has acted at some stage $t <s$ and has not been cancelled (or reset).

At stage $s$, define $\delta_s$ (an approximation to the true path) by recursion as follows. Suppose that $\delta_s \upharpoonright e$ has been defined for $e <s$. Let $\alpha$ be the last node of $\delta_s \upharpoonright e$. We now define $\delta_s \upharpoonright (e+1) \supseteq \delta_s \upharpoonright e$.

\begin{enumerate}
\item If $\alpha$ is a $Q$-node, look to see if it is waiting for a number $x$ to go into $B_\beta$, \ie there is some $x$ in $M_\alpha$ that has not entered $B_\beta$. If so, look to see if an active node $\eta$ to to the left or below $\alpha^\smallfrown \langle \infty \rangle$ would like to put $x$ in (or keep $x$ out). Let $\eta$ act. If there are no remaining numbers that $\alpha$ is waiting on, let $\delta_s \upharpoonright (e+1) = (\delta_s \upharpoonright e)^\smallfrown \langle c \rangle$. Otherwise, look to see if there is some $z$ in the remaining numbers such that $h(z,\Phi_{M_\alpha}(z))$ has converged and our current stage is greater than $h(z,\Phi_{M_\alpha}(z))+1$. If so, let $\delta_s \upharpoonright (e+1) = (\delta_s \upharpoonright e)^\smallfrown \langle \infty \rangle$. If not and some $x$ was pulled, let $\delta_s \upharpoonright (e+1) = (\delta_s \upharpoonright e)^\smallfrown \langle c \rangle$. Otherwise, let $\delta_s \upharpoonright (e+1) = (\delta_s \upharpoonright e)^\smallfrown \langle h \rangle$.
\item If $\alpha$ is a $S_V$-node, look to see if $x_\alpha$ has been defined. If so, let $\delta_s \upharpoonright (e+1) = (\delta_s \upharpoonright e)^\smallfrown \langle k \rangle$. Otherwise, go through steps (1) - (4) in the strategy for $S_V$. If one of (1) - (4) holds and we can define $x_\alpha$, let $\delta_s \upharpoonright (e+1) = (\delta_s \upharpoonright e)^\smallfrown \langle k \rangle$. Otherwise, let $\delta_s \upharpoonright (e+1) = (\delta_s \upharpoonright e)^\smallfrown \langle s \rangle$. 
\item If $\alpha$ is a $D$-node, look to see whether $\Delta_i(x_\alpha)$ has converged and equals 0. If so, let $\delta_s \upharpoonright (e+1) = (\delta_s \upharpoonright e)^\smallfrown \langle a \rangle$. Otherwise, let $\delta_s \upharpoonright (e+1) = (\delta_s \upharpoonright e)^\smallfrown \langle d \rangle$.
\end{enumerate}

Reset nodes to the right of $\delta_s$ and let active $S$-nodes to the left of $\delta_s$ or below the $\infty$ outcome for one of the nodes in $\delta_s$ act. At substage $t \leq s$, let $\delta_s \upharpoonright t$ act according to its description in the strategies above, \ie $\delta_s \upharpoonright t$ enumerates numbers in its sets and extends their definitions.

\subsection*{Verification} Recall that the true path denotes the leftmost path traveled through infinitely often by the construction.

\begin{lemma}
For every node $\alpha$, if $x$ is an unused witness in $P_\beta$ and is not reset as a witness, $x$ does not go into $B$ while $\alpha$ is restraining $x$. For every node $\alpha$, if $x_\alpha = x_\gamma$ for another node $\gamma$, we can successfully keep $x_\alpha$ out of $B$. If $\alpha$ is on the true path, $M_\alpha \subseteq^\ast B$. If $\alpha$ is on the true path, we only reset its witness finitely many times.
\end{lemma}
\begin{proof}
The proof follows the same reasoning as in the one split case.
\end{proof}

\begin{lemma} There cannot be more than three nodes with true outcome $c$ on the true path.
\end{lemma}
\begin{proof}
Suppose otherwise and let $\beta \subseteq \alpha \subseteq \delta \subseteq \gamma$ be the first four nodes on the true path with outcome $c$. By our scheme of switching between $X$ and $Y$, we obtain the following equations: $X_0 \cup X_1 = B$ (in $R_\beta$), $X_0 \cup Y_1 = B$ (in $R_\alpha)$, $Y_0 \cup X_1 = B$ (in $R_\delta)$ and $Y_0 \cup Y_1 = B$ (in $R_\gamma)$. By our choice of $\beta, \alpha, \delta$ and $\gamma$, we have $R_\beta \supseteq R_\alpha \supseteq R_\delta \supseteq R_\gamma$ and thus all of these equations are true in $R_\gamma$. The first two equations give the following equation in $R_\gamma$: $X_0 = B$ and the last two equations give the following equation in $R_\gamma$: $Y_0 = B$. At some stage, the satisfaction of some $D_i$ requirement puts some element $x$ into $B \cap R_\gamma$ as requirements being satisfied by nodes on the true path after the $c$ outcome of $\gamma$ only take witnesses from $R_\gamma$. However, this would mean that $x$ enters both $X_0$ (from the equation in $R_\gamma$: $X_0 = B$) and $Y_0$ (from the equation in $R_\gamma$: $Y_0 = B$). As $X_0$ and $Y_0$ are disjoint, we have obtained our contradiction.
\end{proof}

\begin{lemma} We only switch $Z_0$ and $Z_1$ finitely many times on the true path.
\end{lemma}
\begin{proof} This follows immediately from the previous lemma as the construction does not switch the value of $Z_0$ or $Z_1$ unless it meets a $c$ outcome.
\end{proof}

\begin{lemma} On the true path, for every requirement, there is a node that works on it.
\end{lemma}
\begin{proof} It suffices to show that there exists an $\alpha$ on the true path such that for all $\nu \supseteq \alpha$ on the true path, $L_2(\alpha)$ is empty. Let $\gamma$ be an arbitrary node and let $\delta$ be the successor of $\gamma$. By construction of $L_2$, either $L_2$ is a finite set or is empty. $L_2(\gamma)$ is not empty if and only if $L_2(\delta)$ is not empty or $\delta$ is a node working on a $Q_h$ requirement with outcome $c$. The latter case can only occur less than three times by Lemma 4.3. The former case only occurs finitely many times as the cardinality of the value of $L_2$ strictly decreases as we go down a path unless we met another node working on a $Q_h$ requirement with outcome $c$. However, there are only finitely many such nodes on the true path so the lemma follows by setting $\alpha$ to be the first node such that $L_2(\alpha)$ is empty and no nodes after $\alpha$ has the $c$ outcome on the true path.
\end{proof}

\begin{lemma}
For all $i$, the $D_i$ requirement is satisfied.
\end{lemma}
\begin{proof}
This follows by Lemma 4.2 and similar reasoning as in the one split case.
\end{proof}

\begin{lemma} $S_0$ and $S_1$ are both satisfied.
\end{lemma}
\begin{proof}
If we never switch the value of $Z_0$ to $Y_0$ on the true path, we never reset the $S_{0,V}$ strategies and as there is a node working on $S_{0,V}$ for every $V$ on the true path, the $S_0$ requirement is satisfied. If we do switch, we finish switching at some stage $s$ by Lemma 4.4. We do not reset the value of $Z_0$ for $S_0$ strategies again after we finish switching so there is a node working on $S_{0,V}$ with the final value of $Z_0$ for every $V$ on the true path and thus the $S_0$ requirement is satisfied by Lemma 4.5 and by our scheme of repeating $S_V$-requirements using list $L_2$. 

The same reasoning works for $S_1$ as well. By Lemma 4.4, $Z_1$ finishes switching at some stage $t$ by Lemma 4.4. We do not reset the value of $Z_1$ for $S_1$ strategies again after we finish switching so there is a node working on $S_{1,V}$ with the final value of $Z_1$ for every $V$ on the true path and thus the $S_1$ requirement is satisfied.
\end{proof}

\begin{lemma}  For all $h$, $Q_h$ is satisfied.
\end{lemma}
\begin{proof} This follows from Lemma 4.3. As there can only be three nodes on the true path with outcome $c$, some $\alpha$ on the true path working on $Q_h$ must get the $\infty$ outcome, \ie there are infinitely many elements in $M_\alpha$ that are kept out of $B$ until stage $h(x,\Phi_{M_\alpha}(x)) +1$.
\end{proof}
\end{proof}

\section{The general case}

In this section, we prove Theorem 2.3.

We would like to use the same line of argument as in the two split case where having too many equations (as in the sense of Definition 3.9) leads to a contradiction. Here, we are dealing with infinitely many splits. During the construction, various sets will be proven to be recursive and various switchings will occur. The worry is that this may fill up the whole universe and then there would be no way to make $B$ speedable or even nonrecursive. Lemma 5.5 is devoted to establishing that such a situation does not occur.

\begin{proof} [Proof of Theorem 2.3.]
We deal with all possible splits of $B$. Recursively order all pairs $X_i, Y_i$ of disjoint \re subsets of $B$. As before, we let $Z_i$ be the split we are currently trying to make nonspeedable. To start, $Z_i$ is $X_i$. For each node $\alpha$, the priority tree assigns which $Z_i$ $\alpha$ is working on.
\subsection*{Requirements}
We have $Q_h$ requirements like in the one split case:
\[Q_{h} : (\exists M \subseteq^* B)(\exists^\infty x)(\Phi_B(x) > h(x,\Phi_M(x)) \vee \text{ switch some } Z_i \]

Switching some $Z_i$ is a form of progress on higher priority strategies just as concluding that $Y$ is nonspeedable was a form of progress for the one split case.

We have $S$-requirements to make $Z_i$ nonspeedable:
\[S_{i}: (\exists g_i)(\forall \re V)\text{ either } V \not \subseteq Z_i \text{ or } [x \in Z_i \Rightarrow \Phi_{Z_i}(x) \leq g_i(x,\Phi_V(x))] \]
and their subrequirements, $S_{i,V}$:
\[S_{i,V}: (\exists x \in V) x \notin Z_i \text{ or } (\forall x)(x \in Z_i \Rightarrow \Phi_{Z_i}(x) \leq g_i(x,\Phi_V(x))) \]

We refer to $S_i$ as the parent requirement of $S_{i,V}$.

We define $g_i$ at the parent node working on $S_i$. We also require $S_i$ to have a higher priority than $S_j$ iff $i <j$. We occasionally drop the $i$ subscript if the context is clear.

Again, we have the $D_i$ requirements as before to help the $Q$-requirements.

\subsection*{Strategies}
Let $f_i(x,s)$ be the least $y$ greater than every $h(x,t)$, where $h$ belongs to a higher priority $Q_h$ that does not have the $h$ outcome and greater than every $g_j(x, t)$ where $g_j$ belongs to a higher priority $S_j$ for all $t\leq s$.
Let $g_i(x,s)$ be the least $y$ greater than $f_i(x,s)$ and also greater than the least stage $t$ that $x$ enters either $X_i$ or $Y_i$ if $x$ has entered $B_{f_i(x,s)}$. This definition of $g_i$ is needed in the verification to show that if $\alpha$ is a $Q$-node, child nodes for lower priority $S_i$ requirements cannot injure $\alpha$ (see first paragraph of the proof of Lemma 5.3).

The strategies for $Q_h$, $S_{i,V}$ and $D_i$ are the same as in the one and two split cases except that the $S_{i,V}$ strategy has the following differences:

\begin{enumerate}
\item Each node $\alpha$ has its version of the $Z_i^\alpha$, $f_i^\alpha$, $g_i^\alpha$, determined by the priority tree. Instead of $X$ and $Y$ in the strategy for $S_V$ in the one split case, we work on $Z_i^\alpha$ and the other side of the split for $S_{i,V}$.
\item The $(\sharp$) commitment we use for the general case in the $S_{i,V}$ requirement is:
\begin{equation}
\text{When } x \text{ enters } V \text{ at stage } s \text{, we put } x \text{ into } B \text{ at stage } f_i(x,s). \tag{$\sharp$}
\end{equation}
\end{enumerate}
\subsection*{Switching of the $Z_i^\alpha$ and reduction to the $n$-split case}
At an $Q$-node $\alpha$, we will prove in the verification that any equation obtained in $R_\alpha$ will only involve $Z_i$'s where $i$ is such that the $S_i$ requirement is assigned to a node above $\alpha$. We will only introduce new splits to be considered \emph{only} after $\alpha$'s requirement has been satisfied by a $\infty$ outcome (\ie until some node assigned to $\alpha$'s requirement gets the $\infty$ outcome, we do not assign any new $S_i$-requirement to any node on this path).

Let $\alpha$ be an arbitrary node on the tree working on some $Q_h$-requirement and let $Z_0,...,Z_n$ be a listing of $Z_i$'s such that the $S_i$ requirement is assigned to some node above $\alpha$.

The switching of the $Z_i^\alpha$ is similar to the two split case. A $c$ outcome for a $Q$-node causes a switch of the highest indexed $Z_i$ that is equal to $X_i$ and rests all higher indexed (for $k > i$) sets $Z_k$ and $X_k$. The reasoning behind this switching is to systematically switch so that if we just look at $Z_0,...,Z_l$ and all of the switches have occurred concerning $Z_0,...,Z_l$, we obtain equations $Z_0 \cup ... \cup Z_l = B$ for every combination of values for $Z_0,...,Z_l$. This is necessary for proving that we can make $B$ speedable.

A typical situation is as in Figure 4 below. There are nodes working on the same $Q_h$ requirement and as in the two split case, the $Q_h$ requirement is continually being assigned to nodes along this path until one of the nodes $\gamma$ working on this $Q_h$ requirement gets an $\infty$ outcome. Until $\gamma$ appears, all nodes appearing after $\alpha$ on this path are working on a same fixed number of splits as $\alpha$. If $\gamma$ never appears, the situation for nodes appearing after $\alpha$ is as in the $n$-split case. Like in the 2 split case, we can argue as in Lemma 4.3 that we obtain a contradiction from having too many equations resulting from too many nodes working on the $Q_h$ requirement with outcome $c$. Thus, $\gamma$ has to appear.

\begin{figure}
\begin{tikzpicture}[level distance=2cm, sibling distance=1cm, edge from parent path={(\tikzparentnode) -- (\tikzchildnode)}]
\Tree [.\node {$\alpha$ ($Q_h$ requirement)}; \edge[dotted] node[auto=right] {$\infty$};
    [.\node {$S_{V_1}$}; ]
     \edge node[auto=right] {$c$}; [.\node {$\delta$ (same $Q_h$ requirement)}; \edge[dotted] node[auto=right] {$\infty$}; [.\node {$S_{V_2}$};] \edge node[auto=right] {$c$}; [.\node {$\nu$ (same $Q_h$ requirement)}; \edge[dashed] node[auto=left] {Switching of the $Z_i^\alpha$ occurs as in $n$-case}; [.\node {$\gamma$ (same $Q_h$ requirement)}; \edge node[auto=right] {$\infty$}; [.\node {}; ]\edge[dotted] node[auto=right] {$c$ };[.\node {};]\edge[dotted] node[auto=right] {$h$};[.\node {};]]]]]
\end{tikzpicture}
\caption{Assuming that this figure signifies a portion of the true path, a dotted line signifies that the true path does not go this way and the dashed edges between nodes signifies the true path in between the two nodes. Here, we assume that all $Q$-nodes on the true path between $\nu$ and $\gamma$ that work on the same $Q_h$ requirement as $\alpha$ has true outcome $c$. Technically, there are other nodes working on $S$-requirements between $\alpha$ and $\delta$ and $\delta$ and $\nu$ but we have not included them in the picture.}
\end{figure}
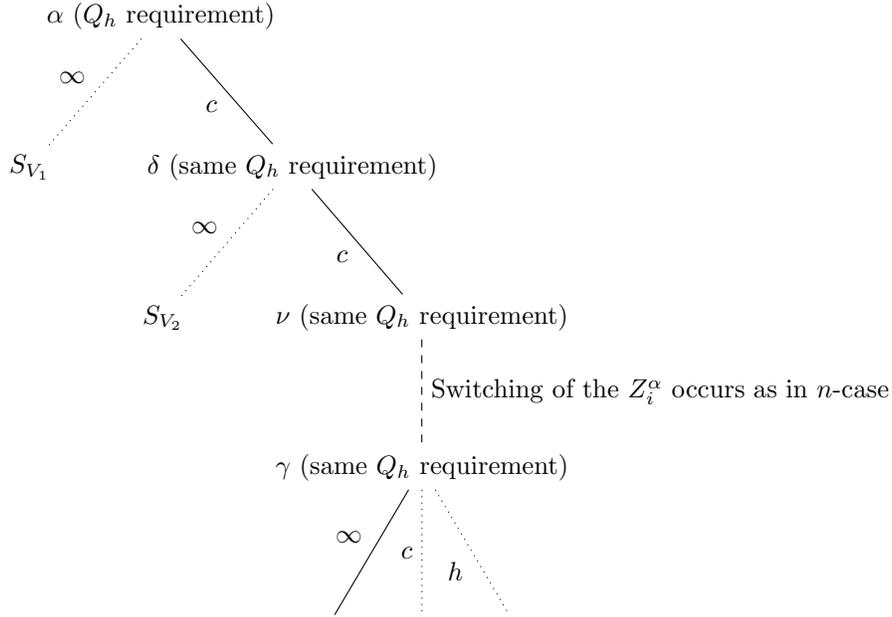
\subsection*{Outcomes}
Each node $\alpha$ working on a $D$-strategy has outcomes: $a$ (for acts) and $d$ (for diverge). On the $a$ outcome, $\Delta_i(x)$ converges and equals to 0 and $\alpha$ would like to put $x$ into $B$. Along the $d$ outcome, either $\Delta_i(x)$ diverges or is not equal to 0.

\medskip

Each node $\alpha$ working on a $Q$-strategy has outcomes: $\infty$, $c$ and $h$. Along the $\infty$ outcome, infinitely many $x$'s in $M_\alpha$ enter $B$ after stage $h(x,\Phi_{M_\alpha}(x))+1$. Along the $c$ outcome, cofinitely many $x$'s in $M_\alpha$ enter $B$ before $h(x,\Phi_{M_\alpha}(x))+1$ and either we have the means to conclude that $Z_i$ is nonspeedable or we switch $Z_i$. Along the $h$ outcome, we see that $h$ is partial.

\medskip

Each node $\alpha$ working on a $S_i$-strategy has outcomes: $split$ or $finite$. Along the $split$ outcome, we see that $X_i$ and $Y_i$ form a split of $B$ and define $f_i$ and $g_i$ for the child-nodes of $S_i$ to work on. Along the $finite$ outcome, we see that $X_i$ and $Y_i$ do not form a split of $B$. Under the $finite$ outcome, there are no child nodes of $S_i$.

Each node $\alpha$ working on a $S_{i,V}$-strategy has outcomes: $k$ and $s$. Along the $k$ outcome, either there is some $x$ in both $V$ and $Y$ or $\alpha$ can keep some element in $V$ out of $B$. Along the $s$ outcome, $\alpha$ cannot achieve $(\exists x)(x \in V \wedge x \notin X)$ and we try to achieve $x \in X \Rightarrow \Phi_X(x) \leq g(x, \Phi_V(x))$ for all $x$.

\subsection*{Priority Tree}
Recursively order $D$-, $Q$- and the $S$-requirements.

Fix a recursive ordering of the $D$-, $Q$- and $S$-requirements where $S_i$ is of higher priority of than $S_{i,V}$.

Let $\Lambda = \{\infty, c, h, split, finite, k,s,a,d\}$ with ordering $\infty < c < h<  split < finite <k <s <a <d$. The tree is a subset of $\Lambda^{<\omega}$.

As in the two split case, we define by recursion a function $G$ such that $G$ assigns to each $\alpha$ the list of $Z_i^\alpha$'s it is working on in the construction (\eg $G(\alpha) = (X_0, Y_1, X_2)$). As in the two split case, we also keep track of two lists $L_1(\alpha)$ and $L_2(\alpha)$ for each node. $L_1$ keeps track of the $S$-requirements that have appeared so far (but gets reset at every appearance of a $c$ outcome) and $L_2$ keeps track of $S$-requirements that need to be repeated.

We assign requirements to nodes as well as define $G, L_1, L_2$ by recursion at the same time.
For the empty node, assign the highest priority $Q_h$ requirement to $\alpha$. Let $G(\emptyset) = (X_0, X_1)$ and let $L_1(\emptyset) = L_2(\emptyset) = \emptyset$. Suppose that we have assigned a requirement to $\beta =\alpha \upharpoonright (|\alpha| -1)$, which we will call the $\beta$ requirement and defined $G(\beta), L_1(\beta)$ and $L_2(\beta)$. We now assign a requirement to $\alpha$ and define $G(\alpha), L_1(\alpha)$ and $L_2(\alpha)$.

\subsubsection*{Defining $G(\alpha)$}
Ask whether $\beta$ is a node working on some $Q_h$ with successor $c$. If so, let $k$ be the largest index such that $X_k$ appears in $G(\beta)$ and let $\overrightarrow{Z_0}$ and $\overrightarrow{Z_1}$ be such that $G(\beta) = \overrightarrow{Z_0} X_k \overrightarrow{Z_1}$. Note that we will prove that $k$ always exists if $\alpha$ is on the true path. If $k$ does not exist, define $G(\alpha)$ to be $G(\beta)$. If $k$ exists, define $G(\alpha)$ to be $G(\beta)$ with $X_k$ switched to $Y_k$ and $X_i$'s and $Y_i$'s in $\overrightarrow{Z_1}$ reset to be $X_i$ \ie $G(\alpha) = \overrightarrow{Z_0}Y_k\overrightarrow{X_1}$ where $\overrightarrow{X_1}$ is the $X$-side of the splits in $\overrightarrow{Z_1}$. If $\beta$ is not working on some $Q_h$ with successor $c$, ask whether $\beta$ is a node working on some $S_i$. If so, let $G(\alpha)$ be $G(\beta)$ with $X_i$ appended to the end of $G(\beta)$'s list. If not, define $G(\alpha)$ to be $G(\beta)$.

\subsubsection*{Defining $L_1(\alpha)$ and $L_2(\alpha)$}
Ask whether $\beta$ is a node working on some $Q_h$ with successor $c$. If so, let $L_2(\alpha) = L_1(\beta)$ and set $L_1(\beta) = \emptyset$. Otherwise, ask whether $\beta$ is a node working on some $S$-requirement. If so, let $L_1(\alpha)$ be $L_1(\beta)$ with $\beta$'s requirement affixed at the end. If $L_2(\alpha)$ is not empty and this $S$-requirement is in $L_2(\alpha)$, let $L_2(\alpha)$ be $L_2(\beta)$ with this $S$ requirement removed. Otherwise, let $L_1(\alpha) = L_1(\beta)$ and let $L_2(\alpha) = L_2(\beta)$.

\subsubsection*{Assigning a requirement to $\alpha$}
If the $\beta$-requirement is an $Q$-requirement with outcome $c$, assign the $\beta$-requirement to $\alpha$ and let its successors be $\infty$, $c$ and $h$.

Otherwise, check if $L_2(\alpha)$ is empty or not. If $L_2(\alpha)$ is not empty, assign the highest priority $S$ requirement in $L_2(\alpha)$ to $\alpha$. If this requirement is a $S_i$-requirement, let its successors be the $split$ and $finite$ outcomes and if this requirement is a $S_V$-requirement, let its successors be $k$ and $s$.

If $L_2(\alpha)$ is empty, we assign a requirement to $\alpha$ based on the type of requirement assigned to $\beta$-requirement: If the $\beta$-requirement is a $Q$-requirement, assign the highest priority $D$-requirement that has not been assigned so far and let its successors be $a$ and $d$. If the $\beta$-requirement is a $D$-requirement, assign the highest priority $S$-requirement that has not been assigned so far and such that $\beta$ does not extend the $finite$ outcome for a node assigned to the parent requirement of this $S$-requirement (if this $S$-requirement is not the parent requirement itself) and let its successors be $k$ and $s$ if it is a $S_V$ requirement and let its successors be $split$ and $finite$ if it is a $S_i$ requirement. Otherwise, assign the highest priority $Q$-requirement that has not been assigned so far and let its successors be $\infty$, $c$ and $h$.
\subsection*{Construction}
We call a node working on a $S$-strategy \emph{active} at stage $s$ if the node has acted at some stage $t <s$ and has not been cancelled (or reset).

At stage $s$, define $\delta_s$ (an approximation to the true path) by recursion as follows. Suppose that $\delta_s \upharpoonright e$ has been defined for $e <s$. Let $\alpha$ be the last node of $\delta_s \upharpoonright e$. We now define $\delta_s \upharpoonright (e+1) \supseteq \delta_s \upharpoonright e$.

\begin{enumerate}
\item If $\alpha$ is a $Q$-node, look to see if it is waiting for a number $x$ to go into $B_\beta$, \ie there is some $x$ in $M_\alpha$ that has not entered $B_\beta$. If so, look to see if an active node $\eta$ to to the left or below $\alpha^\smallfrown \langle \infty \rangle$ would like to put $x$ in (or keep $x$ out). Let $\eta$ act. If there are no remaining numbers that $\alpha$ is waiting on, let $\delta_s \upharpoonright (e+1) = (\delta_s \upharpoonright e)^\smallfrown \langle c \rangle$. Otherwise, look to see if there is some $z$ in the remaining numbers such that $h(z,\Phi_{M_{\alpha}}(z))$ has converged and our current stage is greater than $h(z,\Phi_{M_{\alpha}}(z))+1$. If so, let $\delta_s \upharpoonright (e+1) = (\delta_s \upharpoonright e)^\smallfrown \langle \infty \rangle$. If not and some $x$ was pulled, let $\delta_s \upharpoonright (e+1) = (\delta_s \upharpoonright e)^\smallfrown \langle c \rangle$. Otherwise, let $\delta_s \upharpoonright (e+1) = (\delta_s \upharpoonright e)^\smallfrown \langle h \rangle$.
\item If $\alpha$ is a $S_V$-node, look to see if $x_\alpha$ has been defined. If so, let $\delta_s \upharpoonright (e+1) = (\delta_s \upharpoonright e)^\smallfrown \langle k \rangle$. Otherwise, go through steps (1) - (4) in the strategy for $S_V$. If one of (1) - (4) holds and we can define $x_\alpha$, let $\delta_s \upharpoonright (e+1) = (\delta_s \upharpoonright e)^\smallfrown \langle k \rangle$. Otherwise, let $\delta_s \upharpoonright (e+1) = (\delta_s \upharpoonright e)^\smallfrown \langle s \rangle$.
\item If $\alpha$ is a $S_i$-node, look to see if $X_i$ and $Y_i$ seem to form a split of $B$. If so, let $\delta_s \upharpoonright (e+1) = (\delta_s \upharpoonright e)^\smallfrown \langle split \rangle$. Otherwise, let $\delta_s \upharpoonright (e+1) = (\delta_s \upharpoonright e)^\smallfrown \langle finite \rangle$.
\item If $\alpha$ is a $D$-node, look to see whether $\Delta_i(x_\alpha)$ has converged and equals 0. If so, let $\delta_s \upharpoonright (e+1) = (\delta_s \upharpoonright e)^\smallfrown \langle a \rangle$. Otherwise, let $\delta_s \upharpoonright (e+1) = (\delta_s \upharpoonright e)^\smallfrown \langle d \rangle$.
\end{enumerate}

Reset nodes to the right of $\delta_s$ and let active $S$-nodes to the left of $\delta_s$ or below the $\infty$ outcome for one of the nodes in $\delta_s$ act. At substage $t \leq s$, let $\delta_s \upharpoonright t$ act according to its description in the strategies above, \ie $\delta_s \upharpoonright t$ enumerates numbers in its sets and extends their definitions.
\subsection*{Verification}
Recall that the true path denotes the leftmost path travelled through infinitely often. In the following lemmas, we may drop ``in $R_\alpha$" when referring to an equation. In the proofs below, we only deal with a sequence of equations in $R_\alpha$ where $\alpha$'s are on the same path. As $R_\alpha \subseteq R_\beta$ if $\beta \subseteq \alpha$, these equations in different $R_\alpha$'s are true in their intersection (see Lemma 5.4). We also drop the superscript $\alpha$ when referring to $Z_i^\alpha$ as the version of $Z_i$ we refer to will be clear from context. The first few lemmas lead to proving that there is a node working on every requirement on the true path.

\begin{lemma}
For every node $\alpha$, if $x$ is an unused witness in $P_\beta$ and is not reset as a witness, $x$ does not go into $B$ while $\alpha$ is restraining $x$. For every node $\alpha$, if $x_\alpha = x_\gamma$ for another node $\gamma$, we can successfully keep $x_\alpha$ out of $B$. If $\alpha$ is on the true path, $M_\alpha \subseteq^\ast B$. If $\alpha$ is on the true path, we only reset its witness finitely many times.
\end{lemma}
\begin{proof}
The proof follows the same reasoning as in the one split case.
\end{proof}

The next few lemmas use the following definition.
\begin{definition}
If $\alpha$ working on $Q_h$ has outcome $c$, we are unable to achieve $(\dagger)$ in the $Q_h$ requirement for infinitely many $x$'s. Thus $S_{i,V_j}$ strategies must put cofinitely many of these $x$'s into $B$ before stage $h(x,\Phi_{M_\alpha}(x))+1$ due to $(\sharp)$. Let $S_{i_0, V_{j_0}},\cdots,S_{i_k, V_{j_k}}$ be a listing of the S-strategies involved (\ie the active $S$-nodes below the $\infty$ outcome of $\alpha$). We only commit to $(\sharp)$ when there is no $x$ in both $V$ and the other half of $B$, in which case $V_{j_0} \cup \cdots \cup V_{j_k} = Z_{i_0} \cup \cdots \cup Z_{i_k}$. Therefore, modulo finitely many elements we have the following equation in $R_\alpha$: $B = M = V_{j_0} \cup \cdots \cup V_{j_k} = Z_{i_0} \cup \cdots \cup Z_{i_k}$. We refer to an equation in $R_\alpha$: $B = X_{i_0} \cup \cdots \cup X_{i_j} \cup Y_{k_0} \cup \cdots \cup Y_{k_l}$ as an \emph{equation} obtained from the $c$ outcome if we obtained this equality directly in the way described above. We refer to an equation in $R_\alpha$: $B = X_{i_0} \cup \cdots \cup X_{i_j} \cup Y_{k_0} \cup \cdots \cup Y_{k_l}$ as an \emph{reduced equation} if equations obtained from the $c$ outcome imply it.
\end{definition}

\begin{lemma} Let $\alpha$ be a node working on the $Q_h$ requirement. If $\alpha$'s outcome is $c$, the equation obtained from the $c$ outcome can only involve $S_{i,V}$ requirements where $S_i$ is of higher priority than $Q_h$.
\end{lemma}
\begin{proof} If $S_i$ lies below or to the right of $\alpha$, the lemma follows by our construction of $g_i(x,s)$. We only need $x \in M_\alpha$ to stay out of $B$ before $h(x,\Phi_{M_\alpha}(x)) + 1$ which is strictly less than $g_i(x,\Phi_V(x))$ for $x$'s such that $\Phi_{M_\alpha}(x) \leq \Phi_V(x)$. This is why we need our definition of $g_i$.

Now let $\gamma$ be a node on to the left of $\alpha$ working on a $S_{i,V}$ requirement with its parent node to the left of $\alpha$. We now prove that $S_{i,V}$ cannot be involved in the equation we obtain from $\alpha$'s $c$ outcome. To occur in the equation, $S_{i,V}$ must commit to $(\sharp)$ and some $x \in M_\alpha$ is put in $B$ before $h(x,\Phi_{M_\alpha}(x))+1$ by $S_{i,V}$. $x$ only enters $B_\beta$ (for $\beta \subset \alpha$) after stage $h(x,\Phi_{M_\alpha})+1$ so $x$ has not entered $B_\beta$ at stage $\Phi_V(x)$. By our assumption that $S_{i,V}$ is to the left of $\alpha$ so $S_{i,V}$ would not commit to $(\sharp)$ as $x$ is a potential witness it can keep out.
\end{proof}

In particular, Lemma 5.3 shows that for every node $\alpha$ working on a $Q_h$ outcome, the equations obtained from the $c$ outcome only involves a fixed number of $Z_i$'s.

The next lemma shows that if we have several equations obtained from the $c$ outcome (of $\gamma_i$'s) along some path and $\gamma_i \subseteq \alpha$, then the reduced equation from these equations is true in $R_\alpha$.

\begin{lemma} Let $\alpha$ be on the true path. If we obtain a reduced equation from equations obtained by the $c$ outcome for nodes $\gamma \subseteq \alpha$ then the reduced equation is true of the intersections mentioned in the equation in $R_\alpha$.
\end{lemma}
\begin{proof} Every equation obtained by a $c$ outcome for $\gamma$ is an equation in $R_\gamma$. As $R_\gamma \subseteq R_\alpha$, equations in $R_\gamma$ are also equations in $R_\alpha$. If we have true equations in $R_\alpha$ and we deduce an equation from it, the deduced equation is true in $R_\alpha$. Thus, the reduced equation is true of the intersections mentioned in the reduced equation in $R_\alpha$.
\end{proof}

The following lemma shows that if $\alpha$ is on the true path, in defining $G(\alpha)$, $k$ always exists.

\begin{lemma} Let $\alpha$ be a node working on a $Q_h$ requirement. Let $S_0,...,S_i$ be a listing of all higher priority $S$-requirements. We cannot get an equation from the $c$ outcome of $\alpha$ on the true path of the form $Y_{i_0} \cup \cdots \cup Y_{i_k} = B$ where $\{i_0,...,i_k\} \subseteq \{0,...,i\}$.
\end{lemma}

\begin{proof}
We prove the lemma by induction on $i$.

For $i=0$, suppose by contradiction that we could obtain such an equation, \ie we obtain the equation in $R_\alpha$, $Y_0 = B$. By our scheme of switching between $X$ and $Y$, we must have also obtained the equation in $R_\gamma$: $X_0 = B$ for some $\gamma \subseteq \alpha$. At some stage, the satisfaction of some $D_i$ requirement puts some element $x$ into $B \cap R_\alpha \cap R_\gamma$ as requirements being satisfied by nodes on the true path after the two equations both appear only take witnesses from $R_\alpha \cap R_\gamma$. However, this would mean that $x$ enters both $X_0$ (from the equation in $R_\gamma$: $X_0 = B$ for some $\gamma \subseteq \alpha$) and $Y_0$ (from the equation in $R_\alpha$: $Y_0 = B$ for some $\gamma \subseteq \alpha$). As $X_0$ and $Y_0$ are disjoint, we have obtained our contradiction.

For $i+1$, suppose by contradiction that we could obtain such an equation. By our scheme of switching between $X$ and $Y$, we obtain all equations of length $i+2$ involving all combinations of $Z_0,...,Z_{i+1}$ (from the $c$ outcome), \ie we have $Z_0 \cup ... \cup Z_i \cup X_{i+1} = B$ and $Z_0 \cup...\cup Z_i \cup Y_{i+1} = B$ for every combination of values for $Z_0,...,Z_i$. For a fixed combination of values, the two equations $Z_0 \cup ... \cup Z_i \cup X_{i+1} = B$ and $Z_0 \cup...\cup Z_i \cup Y_{i+1} = B$ imply the reduced equation in $R_\alpha$: $Z_0 \cup ... \cup Z_i = B$ as $X_{i+1}$ and $Y_{i+1}$ are disjoint. As we are considering all combinations of $Z_0,...,Z_i$, we have our contradiction by inductive hypothesis.
\end{proof}

From Lemma 5.5, we see that we are always able to switch one of the $Z_i$'s.

\begin{lemma} For every $i$, we only switch $Z_i$ finitely many times.
\end{lemma}
\begin{proof} By induction on $i$. Let $s$ be the least stage such that all of the higher priority $Z_i$'s do not switch again. We prove that once $Z_i$ switches, it cannot switch back again. Whenever a higher priority $Z_j$ switches, we switch $Z_i$ back to $X_i$ so at stage $s$, $Z_i$ is defined to be $X_i$. If $Z_i$ never switches to $Y_i$, we are done. Otherwise, $Z_i$ switches from $X_i$ to $Y_i$. The only reason why $Z_i$ would switch back to $X_i$ would be because some smaller indexed set switched and thus it cannot switch back after $s$.
\end{proof}

\begin{lemma}
On the true path, for every requirement, there is a node that works on it.
\end{lemma}
\begin{proof} It suffices to show that each $Q_h$ strategy is not repeated infinitely often. Every $Q_h$ strategy is only repeated when it has a $c$ outcome. By Lemma 5.3, we obtain an equation only involving $S_{i,V}$ strategies that come from higher priority $S_i$'s. Thus, whenever we have a $c$ outcome, we have obtain an equation involving a fixed number of $Z$'s: $Z_0,\cdots,Z_k$. By Lemma 5.5, we are always able to switch some $Z$ from an $X$ to a $Y$. By the previous lemma, we can only switch every $Z_i$ finitely many times. Therefore, we cannot have a $c$ outcome occur infinitely often and must go to the $\infty$ outcome.
\end{proof}

\begin{lemma}
For all $i$, the $D_i$ requirement is satisfied.
\end{lemma}
\begin{proof}
This follows by Lemma 5.1 and similar reasoning as in the one split case.
\end{proof}

\begin{lemma} For all $i$, $S_i$ is satisfied.
\end{lemma}
\begin{proof} If we never switch to a $Y_i$ after all of the higher priority $Z_j$'s have finished switching, we do not reset the $S_{i,V}$ strategies. If we do switch, we finish switching at some stage $s$ by Lemma 5.6. We do not reset the $S_i$ strategies again after we finish switching so there is a node working on $S_{i,V}$ for every $V$ on the true path and thus the $S_i$ requirement is satisfied.
\end{proof}

\begin{lemma}  For all $h$, $Q_h$ is satisfied.
\end{lemma}
\begin{proof} By the previous lemmas. Eventually $Z_i$'s stop switching and some $\alpha$ on the true path working on $Q_h$ must get the $\infty$ outcome, \ie there are infinitely many elements in $M_\alpha$ that are kept out of $B$ until stage $h(x,\Phi_{M_\alpha}(x)) +1$.
\end{proof}

\end{proof}

\section{Further generalizations and questions}
One way to generalize Theorem 2.3 is to look generalizations of being semilow (as semilow is equivalent to being nonspeedable). A particularly interesting generalization is that of the notion of semilow$_{1.5}$.

\begin{definition} A \re set $A$ is semilow$_{1.5}$ if and only if
\[ \{e:W_e \cap \overline{A} \text{ infinite }\} \leq_1 \text{Inf}. \]
\end{definition}

Semilow$_{1.5}$ sets occur when studying the lattice of \re sets. Maass \cite{Maass83} showed that if $A$ is cofinite then $A$ is semilow$_{1.5}$ if and only if $\mathcal{L}^*(A) \cong^{\text{eff}} \mathcal{E}^*$ (where $\mathcal{L}(A)$ is the lattice of \re supersets of $A$ and $\mathcal{E}$ is the lattice of \re sets. The $^*$ denotes that we quotient out by the finite sets).

Semilow$_{1.5}$ sets also have a characterization using complexity theoretic notions closely related to nonspeedable sets. Instead of studying the property of having just one a.e. fastest program, Bennison and Soare \cite{MR0479982} defined the notion of a \emph{type 1 c.e. complexity sequence}, which is informally a sequence of lower bounds for all running times of programs for $A$ (with some finite flexibility). They showed that a set has a type 1 c.e. complexity sequence if and only if it is semi-low$_{1.5}$.

One question to examine is whether we can replace semilow with semilow$_{1.5}$ in the statement of the main theorem. In fact, a stronger statement holds \cite{DS1993}. Maximal sets are not semilow$_{1.5}$ but have the property that if $X$ and $Y$ form a split of a maximal set and neither is recursive, then both $X$ and $Y$ are semilow$_{1.5}$. We give a brief proof for completeness.

\begin{lemma} For $B$ maximal, if $X$ and $Y$ form a split of $B$ (and neither is recursive), for every \re $W$, $W-X$ infinite if and only if $W \cap Y$ is infinite.
\end{lemma}
$W \cap Y$ being infinite is a $\Pi^0_2$ property so by Lemma 6.2, $X$ (and by symmetry, $Y$) is semilow$_{1.5}$.

\begin{proof} [Proof of Lemma 6.2] $(\Leftarrow)$ is immediate as $X$ and $Y$ are disjoint.
For the other direction, assume that $W-X$ is infinite for some \re $W$. By maximality, we must have $W \cap \overline{B}$ finite or $\overline{W} \cap \overline{B}$ finite. If we have $W \cap \overline{B}$ finite, then $W \cap Y$ is infinite as $W-X$ is infinite. If we have $\overline{W} \cap \overline{B}$ finite, we show that $W \cap Y$ cannot be finite by contradiction. Suppose that it were finite. Then the complement of $Y$ is equal to $X \cup W \cup (\overline{W} \cap \overline{B})$ minus the finitely many elements in $W \cap Y$. As $X$ and $W$ are \re and $\overline{W} \cap \overline{B}$ is finite, the complement of $Y$ is \re thus $Y$ is recursive, contradicting the assumption that neither $X$ nor $Y$ is recursive.
\end{proof}

The following corollary follows immediately:

\begin{cor} There is a non-semilow$_{1.5}$ set $B$ such that if $X$ and $Y$ form a split of $B$ then at least one of $X$ or $Y$ is semilow$_{1.5}$.
\end{cor}

A further generalization of the notion of being semilow is the notion of being semilow$_2$:

\begin{definition} An \re set $B$ is semilow$_2$ if and only if $\{e:W_e \cap \overline{B} \text{ infinite}\} \leq_T \emptyset^{\prime \prime}.$
\end{definition}

We can ask whether the main theorem generalizes to semilow$_2$ sets:

\begin{quest} Is there a non-semilow$_2$ set $B$ such that if $X$ and $Y$ form a split of $B$ then at least one of $X$ or $Y$ is semilow$_2$?
\end{quest}

Work in progress suggests that the answer to Question 6.5 is yes.

We can also ask the following related question on low sets:

\begin{quest} \cite{DS1993} Is there a non-low r.e. set $B$ such that if $X$ and $Y$ form a non-trivial split of $B$, then both $X$ and $Y$ are low?
\end{quest}

Work in progress suggests that the answer to Question 6.6 is also yes.
\bibliography{ref}
\bibliographystyle{plain}
\end{document}